\documentclass[final]{siamltex}

\usepackage{epsfig,amssymb,latexsym}
\usepackage{amsfonts,psfrag,amsmath,bbm,color}
\usepackage{cancel}
\usepackage{comment}
\usepackage{mathrsfs}
\usepackage{graphicx}
\usepackage{textcomp}
\usepackage{multirow}
\usepackage{enumerate}
\usepackage{cancel}
\usepackage{algpseudocode}
\usepackage{caption}  
\usepackage{subcaption}
\usepackage{url}
\usepackage{rotating}
\usepackage{slashbox}
\usepackage{bigints}
\usepackage{hyperref}
\usepackage{xcolor}
\usepackage{placeins}

\def\grad{{\nabla}}
\newtheorem{algorithm}{Algorithm}[section]
\usepackage{geometry}
\vbadness=\maxdimen

\newcommand{\footremember}[2]{%
    \footnote{#2}
    \newcounter{#1}
    \setcounter{#1}{\value{footnote}}%
}

\usepackage{tikz}
\usetikzlibrary{shapes.geometric, arrows}
\tikzstyle{startstop} = [rectangle, rounded corners, minimum width=1cm, minimum height=1cm,text centered, draw=black]
\tikzstyle{io} = [trapezium, trapezium left angle=70, trapezium right angle=110, minimum width=1cm, minimum height=1cm, text centered, draw=black, fill=blue!30]
\tikzstyle{method} = [rectangle, rounded corners, minimum width=1cm, minimum height =1cm, text centered, draw=black]
\tikzstyle{process} = [rectangle, minimum width=1cm, minimum height=1cm, text centered, draw=black]
\tikzstyle{decision} = [diamond, minimum width=0.5cm, minimum height=0.5cm, text centered, draw=black, fill=green!30]
\tikzstyle{arrow} = [thick,->,>=stealth]

\usepackage{amscd}

\graphicspath{{./figs/}}

\newtheorem{remark}{Remark}[section]

\def\PP{{{\rm l}\kern - .15em {\rm P} }}
\def\PN2{{\PP_{N}-\PP_{N-2}}}

\newcommand{\deleted}[1]{{}}

\begin{document}
\title{High order efficient algorithm for computation of MHD flow ensembles}

\author{
 Muhammad Mohebujjaman\footremember{mit}{D\MakeLowercase{epartment of} M\MakeLowercase{athematics and} P\MakeLowercase{hysics}, T\MakeLowercase{exas} A\&M I\MakeLowercase{nternational} U\MakeLowercase{niversity}, TX 78041, USA; \MakeLowercase{\textcolor{black}{m.mohebujjaman@tamiu.edu}}}%
 \\M\MakeLowercase{assachusetts} I\MakeLowercase{nstitute of} T\MakeLowercase{echnology}}
 
\maketitle

\begin{abstract}
In this paper, we propose, analyze, and test a new fully discrete, efficient, decoupled, stable, and practically second-order time-stepping algorithm for computing MHD ensemble flow averages under uncertainties in the initial conditions and forcing. For each viscosity and magnetic diffusivity pair, the algorithm picks the largest possible parameter $\theta\in[0,1]$ to avoid the instability that arises due to the presence of some explicit viscous terms. At each time step, the algorithm shares the same system matrix with all $J$ realizations but with different right-hand-side vectors. That saves assembling time and computer memory, allows the reuse of the same preconditioner, and can take the advantage of block linear solvers. For the proposed algorithm, we prove stability and convergence rigorously. To illustrate the predicted convergence rates of our analysis, numerical experiments with manufactured solutions are given on a unit square domain. Finally, we test the scheme on a benchmark channel flow over a step problem and it performs well. 
\end{abstract}



\pagestyle{myheadings}
\thispagestyle{plain}
\markboth{\MakeUppercase{High order efficient algorithm for computation of MHD flow ensembles}}{\MakeUppercase{ Muhammad Mohebujjaman}}

\section{Introduction}
Numerical simulations of realistic flows are significantly affected by input data, e.g., initial conditions, boundary conditions, forcing functions, viscosities, etc, which involve uncertainties. As a result, uncertainty quantification (UQ) plays an important role in the validation of simulation methodologies and helps in developing rigorous methods to characterize the effect of the uncertainties on the final quantities of interest. A popular approach for dealing with uncertainties in the data is the computation of an ensemble average of several realizations. In many fluid dynamics applications e.g. ensemble Kalman filter approach, weather forecasting, and sensitivity analyses of solutions \cite{CCDL05, L05, LP08, LK10,  MX06, GG11} require multiple numerical simulations of a flow subject to $J$ different input conditions (realizations), which are then used to compute means and sensitivities.

\textcolor{black}{Recently, the study of MHD flows has become important due to applications in e.g. engineering, physical science, geophysics and astrophysics \cite {BMT07,BLRY07,DS07, F08, HK09, P08}, liquid metal cooling of nuclear reactors \cite{BCL91,H06,SMBM10}, process metallurgy \cite{D01, SM09}, and MHD propulsion\cite{lin1990sea, MG88}.} For the time dependent, viscous and incompressible magnetohydrodynamic (MHD) flow simulations, this leads to solving the following $J$ separate nonlinearly coupled systems of PDEs \cite{B03, D01, LL60, MR17}:
\begin{eqnarray}
u_{j,t}+u_j\cdot\nabla u_j-sB_j\cdot\nabla B_j-\nu \Delta u_j+\nabla p_j &= & f_j(x,t), \hspace{2mm}\text{in}\hspace{2mm}\Omega \times (0,T], \label{gov1}\\
B_{j,t}+u_j\cdot\nabla B_j-B_j\cdot\nabla u_j-\nu_m\Delta B_j+\nabla\lambda_j &= & \nabla\times g_j(x,t)\hspace{2mm}\text{in}\hspace{2mm}\Omega \times (0,T],\label{gov3}\\
\nabla\cdot u_j & =& 0, \hspace{2mm}\text{in}\hspace{2mm}\Omega \times (0,T], \\
\nabla\cdot B_j &=& 0, \hspace{2mm}\text{in}\hspace{2mm}\Omega \times (0,T],\\ 
u_j(x,0)& =& u_j^0(x)\hspace{2mm}\text{in}\hspace{2mm}\Omega,\label{gov2}\\
B_j(x,0)& =& B_j^0(x)\hspace{2mm}\text{in}\hspace{2mm}\Omega.\label{gov5}
\end{eqnarray}
\textcolor{black}{Here,} $u_j$, $B_j$, $p_j$, and $\lambda_j$ denote the velocity, magnetic field, pressure, and artificial magnetic pressure solutions, respectively, of the $j$-{th} member of the ensemble with slightly different initial conditions $u_j^0$ and $B_j^0$, and forcing functions $f_j$ and $\nabla\times g_j$ for all $j=1,2,\cdots, J$. The $\Omega\subset\mathbb{R}^d(d=2\hspace{1mm}\text{or}\hspace{1mm}3)$ is the convex domain, $\nu$ is the kinematic viscosity, $\nu_m$ is the magnetic \textcolor{black}{diffusivity}, $s$ is the coupling number, and $T$ is the simulation time. The artificial magnetic pressure $\lambda_j$ are Lagrange multipliers introduced in the induction equations to enforce divergence free constraints on the discrete induction equations but in continuous case $\lambda_j=0$. All the variables above are dimensionless. \textcolor{black}{The magnetic diffusivity $\nu_m$ is defined by $\nu_m:=Re_m^{-1}=1/(\mu_0\sigma)$, where $\mu_0$ is the magnetic permeability of free space and $\sigma$ is the electric conductivity of the fluid.} For the sake of simplicity of our analysis, we consider homogeneous Dirichlet boundary conditions for both velocity and magnetic fields. For periodic boundary conditions or inhomogeneous Dirichlet boundary conditions, our analyses and results will still work after minor modifications.

 To obtain an accurate, even \textcolor{black}{classical} Navier Stokes (NSE) simulation for a single member of the ensemble, the required number of degrees of freedom (dofs) is very high, which is known from Kolmogorov’s 1941 results \cite{L08}.
 Thus, for a single member of MHD ensemble simulation, where velocity and magnetic fields are nonlinearly coupled, is computationally very expensive with respect to time and memory. As a result, the computational cost of the above \textcolor{black}{coupled} system \eqref{gov1}-\eqref{gov5} will be approximately equal to $J\times$(cost of one MHD simulation) and will generally be computationally infeasible.  Our objective in this paper is to \textcolor{black}{construct} and study an efficient and accurate algorithm for solving the above ensemble systems. It has been shown in recent works \cite{AKMR15, HMR17, MR17,T14} that \textcolor{black}{by} using Els\"asser variables formulation, efficient stable decoupled MHD simulation algorithms can be created. That is, at each time step, instead of solving a fully coupled linear system, two separate Oseen-type problems need to be solved.
  
Defining $v_j:=u_j+\sqrt{s} B_j$, $w_j:=u_j-\sqrt{s}B_j$, $f_{1,j}:=f_{j}+\sqrt{s}\nabla\times g_j$, $f_{2,j}:=f_j-\sqrt{s}\nabla\times g_j$, $q_j:=p_j+\sqrt{s}\lambda_j$ and $r_j:=p_j-\sqrt{s}\lambda_j$ produces the Els{\"{a}}sser variable formulation of the ensemble systems:
\begin{eqnarray}
v_{j,t}+w_j\cdot\nabla v_j-\frac{\nu+\nu_m}{2}\Delta v_j-\frac{\nu-\nu_m}{2}\Delta w_j+\nabla q_j=f_{1,j},\label{els1}\\
w_{j,t}+v_j\cdot\nabla w_j-\frac{\nu+\nu_m}{2}\Delta w_j-\frac{\nu-\nu_m}{2}\Delta v_j+\nabla r_j=f_{2,j},\label{els2}\\
\nabla\cdot v_j=\nabla\cdot w_j=0,\label{els3}
\end{eqnarray}
together with the initial and boundary conditions.

To reduce the ensemble simulation cost, a \textcolor{black}{breakthrough} idea was presented in \cite{JL14} to find a set of $J$ solutions of the NSEs for different initial conditions and body forces. The fundamental idea is that, at each time step, each of the $J$ systems shares a common coefficient matrix, but the right-hand vectors are different.  Thus, \textcolor{black}{the global system matrix needs to be assembled and} the preconditioner needs to \textcolor{black}{be built} only once per time step, and can reuse for all $J$ systems. Also, the algorithm can save storage requirements and take advantage of block linear solvers. This breakthrough idea has been implemented in heat conduction \cite{F17}, Navier-Stokes simulations \cite{J15,jiang2017second,JL15,NTW16}, magnetohydrodynamics \cite{jiang2018efficient,MR17}, parameterized flow problems \cite{GJW18}, and turbulence modeling \cite{JKL15}. In our earlier works in \cite{M17,MR17}, we adopted the idea and considered the first-order scheme with a stabilization term to compute MHD flow ensemble average subject to different initial conditions and forcing functions. In which the optimal convergence in 2D obtained under a mild criteria but in 3D the convergence was proven to be suboptimal. \textcolor{black}{The objective of this paper, is to improve the temporal accuracy of the MHD ensemble scheme.}

\textcolor{black}{It has been shown that in Els\"asser formulation simulations, the second-order approximation of certain viscous terms causes instability unless there is an unusual data restriction $1/2<\nu/\nu_m<2$ \cite{li2018partitioned}.  To overcome this issue, a practically second order $\theta-$scheme is proposed in \cite{HMR17}, where a convex combination of the first and second order approximations of such a viscous term is taken.  We extend this idea together with the efficient ensemble algorithm described above to propose a practically second order timestepping scheme for MHD flow ensemble simulation.}

We consider a uniform timestep size $\Delta t$ and let $t_n=n\Delta t$ for $n=0, 1, \cdots$., for simplicity, we suppress the spatial discretization momentarily. Then computing the $J$ solutions independently, takes the following form:\\ 
Step 1: For $j$=1,...,$J$, 
\begin{align}
\frac{3v_j^{n+1}-4v_j^n+v_j^{n-1}}{2\Delta t}+<w>^n\cdot\nabla v_j^{n+1}+w_j^{'n}\cdot\nabla(2v_j^n-v_j^{n-1})&-\frac{\nu+\nu_m}{2}\Delta v_j^{n+1}\nonumber\\-\theta\frac{\nu-\nu_m}{2}\Delta(2w_j^n-w_j^{n-1})-(1-\theta)\frac{\nu-\nu_m}{2}\Delta w_j^n  +\nabla q_j^{n+1} &= f_{1,j}(t^{n+1}),\label{scheme1}\\\nabla\cdot v_j^{n+1}&=0.\label{incom1}
\end{align}
Step 2: For $j$=1,...,$J$,
\begin{align}
\frac{3w_j^{n+1}-4w_j^n+w_j^{n-1}}{2\Delta t}+<v>^n\cdot\nabla w_j^{n+1}+v_j^{{'}n}\cdot\nabla(2 w_j^n-w_j^{n-1}&)-\frac{\nu+\nu_m}{2}\Delta w_j^{n+1}\nonumber\\-\theta\frac{\nu-\nu_m}{2}\Delta(2v_j^n-v_j^{n-1})-(1-\theta)\frac{\nu-\nu_m}{2}\Delta v_j^n+\nabla r_j^{n+1}  &=f_{2,j}(t^{n+1}),\label{max2}\\\nabla\cdot w_j^{n+1}&=0.\label{scheme2}
\end{align}
Where $v_j^n,w_j^{n}, q_j^n$, and $r_j^{n}$ denote approximations of $v_j(\cdot,t^n),w_j(\cdot,t^n), q_j(\cdot,t^n)$, and $r_j(\cdot,t^n)$, respectively.  \textcolor{black}{The ensemble mean and fluctuation about the mean are denoted by $<u>$, and $u_j^{'}$ respectively, and they are defined as follows: 
\begin{eqnarray}
<u>^n:=\frac{1}{J}\sum\limits_{j=1}^{J}(2u_j^n-u_j^{n-1}), \hspace{2mm} u_j^{'n}:=2u_j^n-u_{j}^{n-1}-<u>^n.\label{ensemble_def}
\end{eqnarray}
}

\noindent We prove that the above method is stable for any $\nu$ and $\nu_m$, provided $\theta$ is chosen to satisfy $\frac{\theta}{1+\theta}<\frac{\nu}{\nu_m}<\frac{1+\theta}{\theta}$, $0\le\theta\le 1$. We also prove that the temporal accuracy of the scheme is of $O(\Delta t^2+(1-\theta)|\nu-\nu_m|\Delta t)$. Though, it looks like the scheme is first order accurate in time but in many practical applications the factor $|\nu-\nu_m|\ll 1$, e.g. in the Earth's core the current estimates suggest  $\nu\sim 10^{-8}$, $\nu_m\sim 10^{-3}$ \cite{jones2015thermal,olson2013experimental}, and $|\nu-\nu_m|$ is in the order of $10^{-3}$. \textcolor{black}{Likewise, the Sun has $\nu_m\sim 10^{-6}$. Moreover, following the discovery of high-temperature superconductor, the theory suggests that there is a possibility of discovering high-temperature liquid superconductor \cite{edwards2006possibility}, where $\nu_m\sim 10^{-10}$.} Thus, in such cases, $\Delta t^2$ dominates over $(1-\theta)|\nu-\nu_m|\Delta t$ and the scheme behaves like second order accurate in time. \textcolor{black}{In fact, when $|\nu-\nu_m|\ll 1$ with a high Reynolds number flow of high electric conductive fluids, the flow becomes convective dominated. In convective dominated regimes, the contribution of the non-linear terms in the system matrix dominates over the contribution of viscous terms. Thus, the system matrix becomes highly ill-conditioned and the system becomes harder for the solver to solve. Therefore, it is critical to find a robust algorithm for high Reynolds number flow of high electric conductive fluids.}

The key features to the efficiency of the above algorithm: (1) The MHD system is decoupled \textcolor{black}{in a stable way} into two \textcolor{black}{identical} Oseen problems and can be solved simultaneously if the computational resources are available. (2) \textcolor{black}{At each time step,} the coefficient matrices of \eqref{scheme1}-\eqref{incom1} and \eqref{max2}-\eqref{scheme2} are independent of $j$. Thus, for each sub-problem, all the $J$ members in the ensemble share the same coefficient matrix. \textcolor{black}{That is, at each time step, instead of solving $J$ individual systems, we only need to solve a single linear system with $J$ different right-hand-side constant vectors.} (3) \textcolor{black}{It provides second-order temporal accuracy when $|\nu-\nu_m|\ll 1$. (4) No data restriction is needed on $\nu$, and $\nu_m$ to avoid instability.}

We give rigorous proofs that the decoupled scheme is conditionally stable and the ensemble average of the $J$ different computed solutions converges to the ensemble average of the $J$ different true solutions, as the timestep size and the spatial mesh width tend to zero. \textcolor{black}{To the best of our knowledge, this second order timestepping scheme to the MHD flow ensemble averages is new.}

This paper is organized as follows.  Section 2 presents notation and mathematical preliminaries which are necessary for a smooth presentation and analysis to follow. In section 3, we present and analyze a fully discrete and decoupled algorithm corresponding to \eqref{scheme1}-\eqref{scheme2}, and prove it's \textcolor{black}{stability and convergent theorems.} Numerical tests are presented in section 4, and finally, conclusions \textcolor{black}{and future directions} are given in section 5.
\section{\large Notation and preliminaries}

Let $\Omega\subset \mathbb{R}^d\ (d=2,3)$ be a convex polygonal or polyhedral domain in $\mathbb{R}^d(d=2,3)$ with boundary $\partial\Omega$. The usual $L^2(\Omega)$ norm and inner product are denoted by $\|.\|$ and $(.,.)$, respectively. Similarly, the $L^p(\Omega)$ norms and the Sobolev $W_p^k(\Omega)$ norms are $\|.\|_{L^p}$ and $\|.\|_{W_p^k}$, respectively for $k\in\mathbb{N},\hspace{1mm}1\le p\le \infty$. Sobolev space $W_2^k(\Omega)$ is represented by $H^k(\Omega)$ with norm $\|.\|_k$. The natural function spaces for our problem are
$$X:=H_0^1(\Omega)=\{v\in (L^p(\Omega))^d :\nabla v\in L^2(\Omega)^{d\times d}, v=0 \hspace{2mm} \mbox{on}\hspace{2mm}   \partial \Omega\},$$
$$Q:=L_0^2(\Omega)=\{ q\in L^2(\Omega): \int_\Omega q\hspace{2mm}dx=0\}.$$
Recall the Poincare inequality holds in $X$: There exists $C$ depending only on $\Omega$ satisfying for all $\phi\in X$,
\[
\| \phi \| \le C \| \nabla \phi \|.
\]
The divergence free velocity space is given by
$$V:=\{v\in X:(\nabla\cdot v, q)=0, \forall q\in Q\}.$$
We define the trilinear form $b:X\times X\times X\rightarrow \mathbb{R}$ by
 \[
 b(u,v,w):=(u\cdot\nabla v,w), 
 \]
and recall from \cite{GR86} that $b(u,v,v)=0$ if $u\in V$, and 
\begin{align}
|b(u,v,w)|\leq C(\Omega)\|\nabla u\|\|\nabla v\|\|\nabla w\|,\hspace{2mm}\mbox{for any}\hspace{2mm}u,v,w\in X.\label{nonlinearbound}
\end{align}

The conforming finite element spaces are denoted by $X_h\subset X$ and  $Q_h\subset Q$, and we assume a regular triangulation $\tau_h(\Omega)$, where $h$ is the maximum triangle diameter.   We assume that $(X_h,Q_h)$ satisfies the usual discrete inf-sup condition
\begin{eqnarray}
\inf_{q_h\in Q_h}\sup_{v_h\in X_h}\frac{(q_h,\grad\cdot v_h)}{\|q_h\|\|\grad v_h\|}\geq\beta>0,\label{infsup}
\end{eqnarray}
where $\beta$ is independent of $h$.

The space of discretely divergence free functions is defined as $$V_h:=\{v_h\in X_h:(\nabla\cdot v_h,q_h)=0,\hspace{2mm}\forall q_h\in Q_h\}.$$

For simplicity of our analysis, we will use Scott-Vogelius (SV) \cite{scott1985conforming} finite element pair $(X_h, Q_h)=((P_k)^d, P_{k-1}^{disc})$,  which satisfies the \textit{inf-sup} condition when the mesh is created as a barycenter refinement of a regular mesh, and the polynomial degree $k\ge d$  \cite{arnold:qin:scott:vogelius:2D,Z05}. Our analysis can be extended without difficulty to any inf-sup stable element choice, {\color{black} however, there will be additional terms that appear in the convergence analysis if non-divergence-free elements are chosen.}

We have the following approximation properties in $(X_h,Q_h)$: \cite{BS08}
\begin{align}
\inf_{v_h\in X_h}\|u-v_h\|&\leq Ch^{k+1}|u|_{k+1},\hspace{2mm}u\in H^{k+1}(\Omega),\label{AppPro1}\\
 \inf_{v_h\in X_h}\|\grad (u-v_h)\|&\leq Ch^{k}|u|_{k+1},\hspace{5mm}u\in H^{k+1}(\Omega),\label{AppPro2}\\
\inf_{q_h\in Q_h}\|p-q_h\|&\leq Ch^k|p|_k,\hspace{10mm}p\in H^k(\Omega),
\end{align}
where $|\cdot|_r$ denotes the $H^r$ seminorm.

We will assume the mesh is sufficiently regular for the inverse inequality to hold, and with this and the LBB assumption, we have approximation properties
\begin{align}
\| \nabla ( u- P_{L^2}^{V_h}(u)  ) \|&\leq Ch^{k}|u|_{k+1},\hspace{2mm}u\in H^{k+1}(\Omega),\label{AppPro3}\\
 \inf_{v_h\in V_h}\|\grad (u-v_h)\|&\leq Ch^{k}|u|_{k+1},\hspace{2mm}u\in H^{k+1}(\Omega),\label{AppPro4}
\end{align}
where $P_{L^2}^{V_h}(u)$ is the $L^2$ projection of $u$ into $V_h$.

\textcolor{black}{The following lemma for the discrete Gronwall inequality was given in \cite{HR90}.
\begin{lemma}
Let $\Delta t$, $H$, $a_n$, $b_n$, $c_n$, $d_n$ be non-negative numbers for $n=1,\cdots, M$ such that
    $$a_M+\Delta t \sum_{n=1}^Mb_n\leq \Delta t\sum_{n=1}^{M-1}{d_na_n}+\Delta t\sum_{n=1}^Mc_n+H\hspace{3mm}\mbox{for}\hspace{2mm}M\in\mathbb{N},$$
then for all $\Delta t> 0,$
$$a_M+\Delta t\sum_{n=1}^Mb_n\leq \mbox{exp}\left(\Delta t\sum_{n=1}^{M-1}d_n\right)\left(\Delta t\sum_{n=1}^Mc_n+H\right)\hspace{2mm}\mbox{for}\hspace{2mm}M\in\mathbb{N}.$$
\end{lemma}}

\section{\large Fully discrete scheme and analysis}

Now we present and analyze an efficient, fully discrete, decoupled, and \textcolor{black}{practically} second-order \textcolor{black}{timestepping} scheme for computing MHD flow ensemble. Like other BDF2 schemes, two initial conditions should be known; if the first initial conditions are known, using the linearized backward Euler scheme in \cite{MR17} without the ensemble eddy viscosity terms on the first timestep, we can get the second initial condition without affecting the stability or accuracy. The scheme is defined below.

\begin{algorithm}\label{Algn1}
Given $\nu$ and $\nu_m$, choose $\theta$ sufficiently large so that $\frac{\theta}{1+\theta}<\frac{\nu}{\nu_m}<\frac{1+\theta}{\theta}$, $0\le\theta\le 1$.
Given time step $\Delta t>0$, end time $T>0$, initial conditions $v_j^0, w_j^0, v_j^1, w_j^1\in V_h$ and $f_{1,j}, f_{2,j}\in L^\infty(0,T;H^{-1}(\Omega)^d)$ for $j=1,2,\cdots, J$. Set $M=T/\Delta t$ and for $n=1,\cdots, M-1$, compute: \\\\
Find $v_{j,h}^{n+1}\in V_h$ satisfying, for all $\chi_h\in V_h$:
\begin{align}
\Bigg(\frac{3v_{j,h}^{n+1}-4v_{j,h}^n+v_{j,h}^{n-1}}{2\Delta t},\chi_h\Bigg)+b^*(<w_h>^n, v_{j,h}^{n+1},\chi_h)+b^*(w_{j,h}^{'n}, 2v_{j,h}^n-v_{j,h}^{n-1},\chi_h)\nonumber\\+\frac{\nu+\nu_m}{2}(\nabla v_{j,h}^{n+1},\nabla \chi_h)+\frac{\nu-\nu_m}{2}((1-\theta)\nabla w_{j,h}^n+\theta\nabla (2w_{j,h}^n-w_{j,h}^{n-1}),\nabla \chi_h)= (f_{1,j}(t^{n+1}),\chi_h),\label{weaknew1}
\end{align}
Find $w_{j,h}^{n+1}\in V_h$ satisfying, for all $l_h\in V_h$:

\begin{eqnarray}
\left(\frac{3w_{j,h}^{n+1}-4w_{j,h}^n+w_{j,h}^{n-1}}{2\Delta t},l_h\right)+b^*(<v_h>^n, w_{j,h}^{n+1},l_h)+b^*(v_{j,h}^{'n},2w_{j,h}^n-w_{j,h}^{n-1},l_h)\nonumber\\+\frac{\nu+\nu_m}{2}(\nabla w_{j,h}^{n+1},\nabla l_h)+\frac{\nu-\nu_m}{2}((1-\theta)\nabla v_{j,h}^n+\theta\nabla (2v_{j,h}^n-v_{j,h}^{n-1}),\nabla l_h)=(f_{2,j}(t^{n+1}),l_h).\label{weaknew2}
\end{eqnarray}
\end{algorithm}

\subsection{\large Stability analysis}
We now prove stability and well-posedness for the Algorithm \ref{Algn1}. To simplify our calculation, we define $\alpha:=\nu+\nu_m-|\nu-\nu_m|(1+2\theta)>0$, which allow us to choose sufficiently large $\theta$ so that
\begin{align}
    \frac{\theta}{1+\theta}<\frac{\nu}{\nu_m}<\frac{1+\theta}{\theta}, \hspace{2mm} 0\le\theta\le 1,\label{theta_con}
\end{align}
holds.
\begin{lemma}
Consider the Algorithm \ref{Algn1}. If the mesh is sufficiently regular so that the inverse inequality holds and the timestep is chosen to satisfy $$\Delta t\le\frac{\alpha h^2}{C\max\limits_{1\le j\le J}\big\{\|\nabla v_{j,h}^{'n}\|^2,\hspace{1mm}\|\nabla w_{j,h}^{'n}\|^2\big\}}$$
then the method is stable and solutions to \eqref{weaknew1}-\eqref{weaknew2} satisfy

\begin{align}
    \|v_{j,h}^M\|^2+\|2v_{j,h}^M-v_{j,h}^{M-1}\|^2&+\|w_{j,h}^M\|^2+\|2w_{j,h}^M-w_{j,h}^{M-1}\|^2+\alpha\Delta t\sum_{n=2}^M\bigg(\|\nabla v_{j,h}^n\|^2+\|\nabla w_{j,h}^n\|^2\bigg)\nonumber\\&\le\|v_{j,h}^1\|^2+\|w_{j,h}^1\|^2+\|2v_{j,h}^1-v_{j,h}^{0}\|^2+\|2w_{j,h}^1-w_{j,h}^{0}\|^2\nonumber\\&+(\nu+\nu_m)\Delta t\bigg(\|\nabla v_{j,h}^0\|^2+\|\nabla w_{j,h}^0\|^2+\|\nabla v_{j,h}^1\|^2+\|\nabla w_{j,h}^1\|^2\bigg)\nonumber\\&+\textcolor{red}{\frac{8}{\alpha}\sum_{n=1}^{M-1}\left(\|f_{1,j}(t^{n+1})\|_{-1}^2+\|f_{2,j}(t^{n+1})\|_{-1}^2\right)}.
\end{align}
\end{lemma}
\textcolor{black}{\begin{remark}
The timestep restriction arises due to the use of inverse inequality, where the constant $C$ depends on the inverse inequality constant. The above stability bound is sufficient for the well-posedness of the Algorithm \ref{Algn1}, since it is linear at each timestep and finite dimensional. The linearity of the scheme also provides the uniqueness of the solution, and the uniqueness implies existence. Thus the solutions to the Algorithm \ref{Algn1} exist uniquely.
\end{remark}}
\vspace{2mm}
\begin{proof}
Choose $\chi_h=v_{j,h}^{n+1}$ and $l_h=w_{j,h}^{n+1}$ in \eqref{weaknew1}-\eqref{weaknew2},
\begin{align}
\Bigg(\frac{3v_{j,h}^{n+1}-4v_{j,h}^n+v_{j,h}^{n-1}}{2\Delta t},v_{j,h}^{n+1}\Bigg)&+\frac{\nu+\nu_m}{2}\|\nabla v_{j,h}^{n+1}\|^2+\frac{\nu-\nu_m}{2}((1+\theta)\nabla w_{j,h}^n-\theta\nabla w_{j,h}^{n-1},\nabla v_{j,h}^{n+1})\nonumber\\&+(w_{j,h}^{'n}\cdot\nabla (2v_{j,h}^n-v_{j,h}^{n-1}),v_{j,h}^{n+1}) = (f_{1,j}(t^{n+1}),v_{j,h}^{n+1}),\label{weakn1}
\end{align}
and
\begin{align}
\Bigg(\frac{3w_{j,h}^{n+1}-4w_{j,h}^n+w_{j,h}^{n-1}}{2\Delta t},w_{j,h}^{n+1}\Bigg)+&\frac{\nu+\nu_m}{2}\|\nabla w_{j,h}^{n+1}\|^2+\frac{\nu-\nu_m}{2}((1+\theta)\nabla v_{j,h}^n-\theta\nabla v_{j,h}^{n-1},\nabla w_{j,h}^{n+1})\nonumber\\&+(v_{j,h}^{'n}\cdot\nabla (2w_{j,h}^n-w_{j,h}^{n-1}),w_{j,h}^{n+1}) = (f_{1,j}(t^{n+1}),w_{j,h}^{n+1}).\label{weakn2}
\end{align}

Using the following identity
\begin{eqnarray}
(3a-4b+c,a)=\frac{a^2+(2a-b)^2}{2}-\frac{b^2+(2b-c)^2}{2}+\frac{(a-2b+c)^2}{2},\label{ident}
\end{eqnarray}
we obtain
\begin{align}
\frac{1}{4\Delta t}\bigg(\|v_{j,h}^{n+1}\|^2-&\|v_{j,h}^n\|^2+\|2v_{j,h}^{n+1}-v_{j,h}^n\|^2-\|2v_{j,h}^n-v_{j,h}^{n-1}\|^2+\|v_{j,h}^{n+1}-2v_{j,h}^n+v_{j,h}^{n-1}\|^2\bigg)\nonumber\\&+\frac{\nu+\nu_m}{2}\|\nabla v_{j,h}^{n+1}\|^2+\frac{\nu-\nu_m}{2}\big((1+\theta)\nabla w_{j,h}^n-\theta\nabla w_{j,h}^{n-1},\nabla v_{j,h}^{n+1}\big)\nonumber\\&+(w_{j,h}^{'n}\cdot\nabla(2v_{j,h}^n-v_{j,h}^{n-1}),v_{j,h}^{n+1})=(f_{1,j}(t^{n+1}),v_{j,h}^{n+1}),\label{MHW1}
\end{align}
and
\begin{align}
\frac{1}{4\Delta t}\bigg(\|w_{j,h}^{n+1}\|^2-&\|w_{j,h}^n\|^2+\|2w_{j,h}^{n+1}-w_{j,h}^n\|^2-\|2w_{j,h}^n-w_{j,h}^{n-1}\|^2+\|w_{j,h}^{n+1}-2w_{j,h}^n+w_{j,h}^{n-1}\|^2\bigg)\nonumber\\&+\frac{\nu+\nu_m}{2}\|\nabla w_{j,h}^{n+1}\|^2+\frac{\nu-\nu_m}{2}\big((1+\theta)\nabla  v_{j,h}^n-\theta\nabla v_{j,h}^{n-1},\nabla w_{j,h}^{n+1}\big)\nonumber\\&+(v_{j,h}^{'n}\cdot\nabla(2w_{j,h}^n-w_{j,h}^{n-1}),w_{j,h}^{n+1})=(f_{2,j}(t^{n+1}),w_{j,h}^{n+1}).\label{MHW2}
\end{align}
Next, using
\begin{align*}
(w_{j,h}^{'n}&\cdot\nabla(2v_{j,h}^n-v_{j,h}^{n-1}),v_{j,h}^{n+1})=(w_{j,h}^{'n}\cdot\nabla v_{j,h}^{n+1},v_{j,h}^{n+1}-2v_{j,h}^n+v_{j,h}^{n-1})\\&\le C\|\nabla w_{j,h}^{'n}\|\|\nabla v_{j,h}^{n+1}\|\hspace{1mm}\|\nabla (v_{j,h}^{n+1}-2v_{j,h}^n+v_{j,h}^{n-1})\|\\&\le \frac{C}{h} \|\nabla w_{j,h}^{'n}\|\|\nabla v_{j,h}^{n+1}\|\hspace{1mm}\|v_{j,h}^{n+1}-2v_{j,h}^n+v_{j,h}^{n-1}\|,
\end{align*}
adding equations \eqref{MHW1} and \eqref{MHW2}, applying Cauchy-Schwarz and Young's inequalities to the $\nu-
\nu_m$ terms, and rearranging, yields
\begin{align*}
\frac{1}{4\Delta t}\bigg(\|v_{j,h}^{n+1}\|^2-\|v_{j,h}^n\|^2+\|2v_{j,h}^{n+1}-v_{j,h}^n\|^2-\|2v_{j,h}^n-v_{j,h}^{n-1}\|^2+\|v_{j,h}^{n+1}-2v_{j,h}^n+v_{j,h}^{n-1}\|^2\\+\|w_{j,h}^{n+1}\|^2-\|w_{j,h}^n\|^2+\|2w_{j,h}^{n+1}-w_{j,h}^n\|^2-\|2w_{j,h}^n-w_{j,h}^{n-1}\|^2+\|w_{j,h}^{n+1}-2w_{j,h}^n+w_{j,h}^{n-1}\|^2\bigg)\\+\frac{\nu+\nu_m}{2}\big(\|\nabla v_{j,h}^{n+1}\|^2+\|\nabla w_{j,h}^{n+1}\|^2\big)\le \frac{|\nu-\nu_m|}{4}(1+2\theta)\big(\|\nabla v_{j,h}^{n+1}\|^2+\|\nabla w_{j,h}^{n+1}\|^2\big)\\+\frac{|\nu-\nu_m|}{4}(1+\theta)\big(\|\nabla v_{j,h}^n\|^2+\|\nabla w_{j,h}^n\|^2\big)+\frac{|\nu-\nu_m|}{4}\theta\big(\|\nabla v_{j,h}^{n-1}\|^2+\|\nabla w_{j,h}^{n-1}\|^2\big)\\ +\frac{C}{h}\|\nabla w_{j,h}^{'n}\|\|\nabla v_{j,h}^{n+1}\|\hspace{1mm}\|v_{j,h}^{n+1}-2v_{j,h}^n+v_{j,h}^{n-1}\|+\frac{C}{h}\|\nabla v_{j,h}^{'n}\|\|\nabla w_{j,h}^{n+1}\|\hspace{1mm}\|w_{j,h}^{n+1}-2w_{j,h}^n+w_{j,h}^{n-1}\|\\+\|f_{1,j}(t^{n+1})\|_{-1}\|\nabla v_{j,h}^{n+1}\|+\|f_{2,j}(t^{n+1})\|_{-1}\|\nabla w_{j,h}^{n+1}\|.
\end{align*}
Next, we apply Young’s inequality using $\alpha/8$ with the forcing and non-linear terms, noting that $\alpha>0$ by the assumed choice of $\theta$,  and hiding terms on the left,
\begin{align*}
\frac{1}{4\Delta t}\bigg(\|v_{j,h}^{n+1}\|^2-\|v_{j,h}^n\|^2+\|2v_{j,h}^{n+1}-v_{j,h}^n\|^2-\|2v_{j,h}^n-v_{j,h}^{n-1}\|^2+\|v_{j,h}^{n+1}-2v_{j,h}^n+v_{j,h}^{n-1}\|^2\\+\|w_{j,h}^{n+1}\|^2-\|w_{j,h}^n\|^2+\|2w_{j,h}^{n+1}-w_{j,h}^n\|^2-\|2w_{j,h}^n-w_{j,h}^{n-1}\|^2+\|w_{j,h}^{n+1}-2w_{j,h}^n+w_{j,h}^{n-1}\|^2\bigg)\\+\frac{\nu+\nu_m}{4}\big(\|\nabla v_{j,h}^{n+1}\|^2+\|\nabla w_{j,h}^{n+1}\|^2\big)\le \frac{|\nu-\nu_m|}{4}(1+\theta)\big(\|\nabla v_{j,h}^n\|^2+\|\nabla w_{j,h}^n\|^2\big)
\\+\frac{|\nu-\nu_m|}{4}\theta\big(\|\nabla v_{j,h}^{n-1}\|^2+\|\nabla w_{j,h}^{n-1}\|^2\big)+\frac{2}{\alpha}\|f_{1,j}(t^{n+1})\|_{-1}^2+\frac{2}{\alpha}\|f_{2,j}(t^{n+1})\|_{-1}^2\\ +\frac{C}{\alpha h^2}\|\nabla w_{j,h}^{'n}\|^2\|v_{j,h}^{n+1}-2v_{j,h}^n+v_{j,h}^{n-1}\|^2+\frac{C}{\alpha h^2}\|\nabla v_{j,h}^{'n}\|^2\|w_{j,h}^{n+1}-2w_{j,h}^n+w_{j,h}^{n-1}\|^2.
\end{align*}

Rearranging

\begin{align}
\frac{1}{4\Delta t}\bigg(\|v_{j,h}^{n+1}\|^2-\|v_{j,h}^n\|^2+\|2v_{j,h}^{n+1}-v_{j,h}^n\|^2-\|2v_{j,h}^n-v_{j,h}^{n-1}\|^2\nonumber\\+\|w_{j,h}^{n+1}\|^2-\|w_{j,h}^n\|^2+\|2w_{j,h}^{n+1}-w_{j,h}^n\|^2-\|2w_{j,h}^n-w_{j,h}^{n-1}\|^2\bigg)\nonumber\\+\bigg(\frac{1}{4\Delta t}-\frac{C}{\alpha h^2}\|\nabla w_{j,h}^{'n}\|^2\bigg)\|v_{j,h}^{n+1}-2v_{j,h}^n+v_{j,h}^{n-1}\|^2\nonumber\\+\bigg(\frac{1}{4\Delta t}-\frac{C}{\alpha h^2}\|\nabla v_{j,h}^{'n}\|^2\bigg)\|w_{j,h}^{n+1}-2w_{j,h}^n+w_{j,h}^{n-1}\|^2\nonumber\\+\frac{\nu+\nu_m}{4}\bigg(\|\nabla v_{j,h}^{n+1}\|^2-\|\nabla v_{j,h}^{n}\|^2+\|\nabla w_{j,h}^{n+1}\|^2-\|\nabla w_{j,h}^{n}\|^2\bigg)\nonumber\\+ \frac{\nu+\nu_m-|\nu-\nu_m|(1+\theta)}{4}\bigg(\|\nabla v_{j,h}^n\|^2-\|\nabla v_{j,h}^{n-1}\|^2+\|\nabla w_{j,h}^n\|^2-\|\nabla w_{j,h}^{n-1}\|^2\bigg)\nonumber\\+\frac{\nu+\nu_m-|\nu-\nu_m|(1+2\theta)}{4}\bigg(\|\nabla v_{j,h}^{n-1}\|^2+\|\nabla w_{j,h}^{n-1}\|^2\bigg)\nonumber\\\le\frac{2}{\alpha}\|f_{1,j}(t^{n+1})\|_{-1}^2+\frac{2}{\alpha}\|f_{2,j}(t^{n+1})\|_{-1}^2.
\end{align}

Now, if we choose $\Delta t\le\frac{\alpha h^2}{C\max\limits_{1\le j\le J}\big\{\|\nabla v_{j,h}^{'n}\|^2,\hspace{1mm}\|\nabla w_{j,h}^{'n}\|^2\big\}}$, multiplying both sides by $4\Delta t$, summing over timesteps $n=1,\cdots,M-1$, and finally dropping non-negative terms from left hand side finish the proof.
\end{proof}

\subsection{\large Error analysis}\label{ErrorAnalysis}
Now we consider the convergence of the proposed decoupled scheme.
\begin{theorem}
Suppose $(v_j,w_j,q_j,r_j)$ satisfies \eqref{els1}-\eqref{els3} with the regularity assumptions $v_j$,$w_j$ $\in L^{\infty}(0, T;$ $H^{m}(\Omega)^d)$ for $m=\max\{2,k+1\}$, $v_{j,t}, w_{j,t}, v_{j,tt}, w_{j,tt}\in L^{\infty}(0,T;H^1(\Omega)^d)$, and $v_{j,ttt}, w_{j,ttt}\in L^{\infty}(0,T;L^2(\Omega)^d)$. Then the ensemble average solution $(<v_{h}>, <w_{h}>)$ to Algorithm \eqref{Algn1} converges to the true ensemble average solution: For any $$\Delta t\le\frac{\alpha h^2}{C\max\limits_{1\le j\le J}\{\|\nabla v_{j,h}^{'n}\|,\|\nabla w_{j,h}^{'n}\|\}},$$ one has
\begin{align}
\|<v>^T-&<v_h>^M\|+\|<w>^T-<w_h>^M\|+\alpha\Delta t\sum\limits_{n=2}^{M}\bigg\{\|\nabla(<v>(t^n)-<v_h>^n)\|^2\nonumber\\&+\|\nabla(<w>(t^n)-<w_h>^n)\|^2\bigg\}^\frac12\le C (h^{k}+(\Delta t)^2+|\nu-\nu_m|(1-\theta)\Delta t).
\end{align}
\end{theorem}
\begin{proof}
We start our proof by obtaining the error equations. Testing \eqref{els1} and \eqref{els2} with $\chi_h, l_h\in V_h$ at the time level $t^{n+1}$, the continuous variational formulations can be written as
\begin{align}
\bigg(\frac{3v_j(t^{n+1})-4v_j(t^n)+v_j(t^{n-1})}{2\Delta t},\chi_h\bigg)+(w_j(t^{n+1})\cdot\nabla v_j(t^{n+1}),\chi_h)\nonumber\\+\frac{\nu+\nu_m}{2}(\nabla v_j(t^{n+1}), \nabla\chi_h) +\frac{\nu-\nu_m}{2}((1+\theta)\nabla w_j(t^n)-\theta\nabla  w_j(t^{n-1}),\chi_h)\nonumber\\ =(f_{1,j}(t^{n+1}),\chi_h)-\frac{\nu-\nu_m}{2}\big(\nabla( w_j(t^{n+1})-(1+\theta)w_j(t^n)+\theta w_j(t^{n-1})),\chi_h\big)\nonumber\\-\bigg(v_{j,t}(t^{n+1})-\frac{3v_j(t^{n+1})-4v_j(t^n)+v_j(t^{n-1})}{2\Delta t}, \chi_h\bigg), \label{conweakn1}
\end{align}
and
\begin{align}
\bigg(\frac{3w_j(t^{n+1})-4w_j(t^n)+w_j(t^{n-1})}{2\Delta t},l_h\bigg)+(v_j(t^{n+1})\cdot\nabla w_j(t^{n+1}),l_h)\nonumber\\+\frac{\nu+\nu_m}{2}(\nabla w_j(t^{n+1}), \nabla l_h) +\frac{\nu-\nu_m}{2}((1+\theta)\nabla v_j(t^n)-\theta\nabla v_j(t^{n-1})),l_h)\nonumber\\ =(f_{2,j}(t^{n+1}),l_h)-\frac{\nu-\nu_m}{2}\big(\nabla( v_j(t^{n+1})-(1+\theta)v_j(t^n)+\theta v_j(t^{n-1})),l_h\big)\nonumber\\-\bigg(w_{j,t}(t^{n+1})-\frac{3w_j(t^{n+1})-4w_j(t^n)+w_j(t^{n-1})}{2\Delta t}, l_h\bigg). \label{conweakn2}
\end{align}
Denote $e_{v,j}^n:=v_j(t^n)-v_{j,h}^n,\hspace{2mm}e_{w,j}^n:=w_j(t^n)-w_{j,h}^n.$ Subtracting \eqref{weaknew1} and \eqref{weaknew2} from equation \eqref{conweakn1} and \eqref{conweakn2} respectively, yields 

\begin{align}
\bigg(\frac{3e_{j,v}^{n+1}-4e_{j,v}^n+e_{j,v}^{n-1}}{2\Delta t},&\chi_h\bigg)+\frac{\nu+\nu_m}{2}(\nabla e_{j,v}^{n+1},\nabla \chi_h)+\frac{\nu-\nu_m}{2}\big((1+\theta)\nabla e_{j,w}^n-\theta\nabla e_{j,w}^{n-1},\nabla\chi_h\big)\nonumber\\&+((2e_{j,w}^n-e_{j,w}^{n-1})\cdot\nabla v_j(t^{n+1}),\chi_h)+((2w_{j,h}^n-w_{j,h}^{n-1})\cdot\nabla e_{j,v}^{n+1},\chi_h)\nonumber\\&-(w_{j,h}^{'n}\cdot\nabla(e_{j,v}^{n+1}-2e_{j,v}^n+e_{j,v}^{n-1}),\chi_h)=-G_1(t,v_j,w_j,\chi_h),
\end{align}
and
\begin{align}
\bigg(\frac{3e_{j,w}^{n+1}-4e_{j,w}^n+e_{j,w}^{n-1}}{2\Delta t},&l_h\bigg)+\frac{\nu+\nu_m}{2}(\nabla e_{j,w}^{n+1},\nabla l_h)+\frac{\nu-\nu_m}{2}\big((1+\theta)\nabla e_{j,v}^n-\theta\nabla e_{j,v}^{n-1},\nabla l_h\big)\nonumber\\&+((2e_{j,v}^n-e_{j,v}^{n-1})\cdot\nabla w_j(t^{n+1}),l_h)+((2v_{j,h}^n-v_{j,h}^{n-1})\cdot\nabla e_{j,w}^{n+1},l_h)\nonumber\\&-(v_{j,h}^{'n}\cdot\nabla(e_{j,w}^{n+1}-2e_{j,w}^n+e_{j,w}^{n-1}),l_h)=-G_2(t,v_j,w_j,l_h),
\end{align}
where
\begin{align*}
G_1(t,v_j,w_j,\chi_h):=&((w_j(t^{n+1})-2w_j(t^n)+w_j(t^{n-1}))\cdot\nabla v_j(t^{n+1}),\chi_h)\nonumber\\&+(w_{j,h}^{'n}\cdot\nabla(v_j(t^{n+1})-2v_j(t^n)+v_j(t^{n-1})),\chi_h)\nonumber\\&+\frac{\nu-\nu_m}{2}\big(\nabla( w_j(t^{n+1})-(1+\theta)w_j(t^n)+\theta w_j(t^{n-1})),\nabla\chi_h\big)\nonumber\\&+\bigg(v_{j,t}(t^{n+1})-\frac{3v_j(t^{n+1})-4v_j(t^n)+v_j(t^{n-1})}{2\Delta t}, \chi_h\bigg),
\end{align*}
and
\begin{align*}
G_2(t,v_j,w_j,l_h):=&((v_j(t^{n+1})-2v_j(t^n)+v_j(t^{n-1}))\cdot\nabla w_j(t^{n+1}),l_h)\nonumber\\&+(v_{j,h}^{'n}\cdot\nabla(w_j(t^{n+1})-2w_j(t^n)+w_j(t^{n-1})),l_h)\nonumber\\&+\frac{\nu-\nu_m}{2}\big(\nabla( v_j(t^{n+1})-(1+\theta)v_j(t^n)+\theta v_j(t^{n-1})),\nabla l_h\big)\nonumber\\&+\bigg(w_{j,t}(t^{n+1})-\frac{3w_j(t^{n+1})-4w_j(t^n)+w_j(t^{n-1})}{2\Delta t}, l_h\bigg).
\end{align*}
Now we decompose the errors as
\begin{align*}
    e_{j,v}^n:& = v_j(t^n)-v_{j,h}^n=(v_j(t^n)-\tilde{v}_j^n)-(v_{j,h}^n-\tilde{v}_j^n):=\eta_{j,v}^n-\phi_{j,h}^n,\\
    e_{j,w}^n: &= w_j(t^n)-w_{j,h}^n=(w_j(t^n)-\tilde{w}_j^n)-(w_{j,h}^n-\tilde{w}_j^n):=\eta_{j,w}^n-\psi_{j,h}^n,
\end{align*}
where $\tilde{v}_j^n: =P_{V_h}^{L^2}(v_j(t^n))\in V_h$ and $\tilde{w}_j^n: =P_{V_h}^{L^2}(w_j(t^n))\in V_h$ are the $L^2$ projections of $v_j(t^n)$ and $w_j(t^n)$ into $V_h$, respectively. Note that $(\eta_{j,v}^n,v_h)=(\eta_{j,w}^n,v_h)=0\hspace{2mm} \forall v_h\in V_h.$  Rewriting, we have for $\chi_h, l_h\in V_h$
\begin{align}
\bigg(&\frac{3\phi_{j,h}^{n+1}-4\phi_{j,h}^n+\phi_{j,h}^{n-1}}{2\Delta t},\chi_h\bigg)+\frac{\nu+\nu_m}{2}(\nabla \phi_{j,h}^{n+1},\nabla \chi_h)+\frac{\nu-\nu_m}{2}\big( (1+\theta)\nabla\psi_{j,h}^n-\theta\nabla\psi_{j,h}^{n-1},\nabla\chi_h\big)\nonumber\\&+((2\psi_{j,h}^n-\psi_{j,h}^{n-1})\cdot\nabla v_j(t^{n+1}),\chi_h)+((2w_{j,h}^n-w_{j,h}^{n-1})\cdot\nabla\phi_{j,h}^{n+1},\chi_h)\nonumber\\&-(w_{j,h}^{'n}\cdot\nabla(\phi_{j,h}^{n+1}-\phi_{j,h}^n+\phi_{j,h}^{n-1}),\chi_h)=\frac{\nu-\nu_m}{2}\big((1+\theta)\nabla \eta_{j,w}^n-\theta\nabla\eta_{j,w}^{n-1},\nabla\chi_h\big)\nonumber\\&+\frac{\nu+\nu_m}{2}(\nabla \eta_{j,v}^{n+1},\nabla \chi_h)+((2\eta_{j,w}^n-\eta_{j,w}^{n-1})\cdot\nabla v_j(t^{n+1}),\chi_h)+((2w_{j,h}^n-w_{j,h}^{n-1})\cdot\nabla\eta_{j,v}^{n+1},\chi_h)\nonumber\\&-(w_{j,h}^{'n}\cdot\nabla(\eta_{j,v}^{n+1}-\eta_{j,v}^n+\eta_{j,v}^{n-1}),\chi_h)+G_1(t,v_j,w_j,\chi_h),\label{phi2n}
\end{align}
and
\begin{align}
\bigg(&\frac{3\psi_{j,h}^{n+1}-4\psi_{j,h}^n+\psi_{j,h}^{n-1}}{2\Delta t},l_h\bigg)+\frac{\nu+\nu_m}{2}(\nabla \psi_{j,h}^{n+1},\nabla l_h)+\frac{\nu-\nu_m}{2}\big((1+\theta)\nabla \phi_{j,h}^n-\theta\nabla\phi_{j,h}^{n-1},\nabla l_h\big)\nonumber\\&+((2\phi_{j,h}^n-\phi_{j,h}^{n-1})\cdot\nabla w_j(t^{n+1}),l_h)+((2v_{j,h}^n-v_{j,h}^{n-1})\cdot\nabla\psi_{j,h}^{n+1},l_h)\nonumber\\&-(v_{j,h}^{'n}\cdot\nabla(\psi_{j,h}^{n+1}-\psi_{j,h}^n+\psi_{j,h}^{n-1}),l_h)=\frac{\nu-\nu_m}{2}\big((1+\theta)\nabla \eta_{j,v}^n-\theta\nabla\eta_{j,v}^{n-1}),\nabla l_h\big)\nonumber\\&+\frac{\nu+\nu_m}{2}(\nabla \eta_{j,w}^{n+1},\nabla l_h)+((2\eta_{j,v}^n-\eta_{j,v}^{n-1})\cdot\nabla w_j(t^{n+1}),l_h)+((2v_{j,h}^n-v_{j,h}^{n-1})\cdot\nabla\eta_{j,w}^{n+1},l_h)\nonumber\\&-(v_{j,h}^{'n}\cdot\nabla(\eta_{j,w}^{n+1}-\eta_{j,w}^n+\eta_{j,w}^{n-1}),l_h)+G_2(t,v_j,w_j,l_h).\label{psi2n}
\end{align}

Choose $\chi_h=\phi_{j,h}^{n+1}, l_h=\psi_{j,h}^{n+1}$ and use the identity \eqref{ident} in \eqref{phi2n} and \eqref{psi2n}, to obtain

\begin{align}
\frac{1}{4\Delta t}(\|\phi_{j,h}^{n+1}\|^2-\|\phi_{j,h}^{n}\|^2+\|2\phi_{j,h}^{n+1}-\phi_{j,h}^{n}\|^2-\|2\phi_{j,h}^{n}-\phi_{j,h}^{n-1}\|^2\nonumber\\+\|\phi_{j,h}^{n+1}-2\phi_{j,h}^{n}+\phi_{j,h}^{n-1}\|^2)+\frac{\nu+\nu_m}{2}\|\nabla \phi_{j,h}^{n+1}\|^2\nonumber\\\le(1+\theta)\frac{|\nu-\nu_m|}{2}\bigg\{|\big(\nabla\psi_{j,h}^n,\nabla\phi_{j,h}^{n+1}\big)|+|\big(\nabla\eta_{j,w}^n,\nabla\phi_{j,h}^{n+1}\big)|\bigg\}\nonumber\\+\theta\frac{|\nu-\nu_m|}{2}\bigg\{|\big(\nabla\psi_{j,h}^{n-1},\nabla\phi_{j,h}^{n+1}\big)|+|\big(\nabla\eta_{j,w}^{n-1},\nabla\phi_{j,h}^{n+1}\big)|\bigg\}+\frac{\nu+\nu_m}{2}|\big(\nabla \eta_{j,v}^{n+1},\nabla \phi_{j,h}^{n+1}\big)|\nonumber\\+|\big((2\eta_{j,w}^n-\eta_{j,w}^{n-1})\cdot\nabla v_j(t^{n+1}),\phi_{j,h}^{n+1}\big)|+|\big((2w_{j,h}^n-w_{j,h}^{n-1})\cdot\nabla\eta_{j,v}^{n+1},\phi_{j,h}^{n+1}\big)|\nonumber\\+|\big(w_{j,h}^{'n}\cdot\nabla(\eta_{j,v}^{n+1}-\eta_{j,v}^n+\eta_{j,v}^{n-1}),\phi_{j,h}^{n+1}\big)|+|\big((2\psi_{j,h}^n-\psi_{j,h}^{n-1})\cdot\nabla v_j(t^{n+1}),\phi_{j,h}^{n+1}\big)|\nonumber\\+|\big(w_{j,h}^{'n}\cdot\nabla (\phi_{j,h}^{n+1}-\phi_{j,h}^n+\phi_{j,h}^{n-1}),\phi_{j,h}^{n+1}\big)|+|G_1(t,v_j,w_j,\phi_{j,h}^{n+1})|,\label{baddddn1}
\end{align}
and
\begin{align}
\frac{1}{4\Delta t}(\|\psi_{j,h}^{n+1}\|^2-\|\psi_{j,h}^{n}\|^2+\|2\psi_{j,h}^{n+1}-\psi_{j,h}^{n}\|^2-\|2\psi_{j,h}^{n}-\psi_{j,h}^{n-1}\|^2\nonumber\\+\|\psi_{j,h}^{n+1}-2\psi_{j,h}^{n}+\psi_{j,h}^{n-1}\|^2)+\frac{\nu+\nu_m}{2}\|\nabla \psi_{j,h}^{n+1}\|^2\nonumber\\\le(1+\theta)\frac{|\nu-\nu_m|}{2}\bigg\{|\big(\nabla\phi_{j,h}^n,\nabla\psi_{j,h}^{n+1}\big)|+|\big(\nabla\eta_{j,v}^n,\nabla\psi_{j,h}^{n+1}\big)|\bigg\}\nonumber\\+\theta\frac{|\nu-\nu_m|}{2}\bigg\{|\big(\nabla\phi_{j,h}^{n-1},\nabla\psi_{j,h}^{n+1}\big)|+|\big(\nabla\eta_{j,v}^{n-1},\nabla\psi_{j,h}^{n+1}\big)|\bigg\}+\frac{\nu+\nu_m}{2}|\big(\nabla \eta_{j,w}^{n+1},\nabla \psi_{j,h}^{n+1}\big)|\nonumber\\+|\big((2\eta_{j,v}^n-\eta_{j,v}^{n-1})\cdot\nabla w_j(t^{n+1}),\psi_{j,h}^{n+1}\big)|+|\big((2v_{j,h}^n-v_{j,h}^{n-1})\cdot\nabla\eta_{j,w}^{n+1},\psi_{j,h}^{n+1}\big)|\nonumber\\+|\big(v_{j,h}^{'n}\cdot\nabla(\eta_{j,w}^{n+1}-\eta_{j,w}^n+\eta_{j,w}^{n-1}),\psi_{j,h}^{n+1}\big)|+|\big((2\phi_{j,h}^n-\phi_{j,h}^{n-1})\cdot\nabla w_j(t^{n+1}),\psi_{j,h}^{n+1}\big)|\nonumber\\+|\big(v_{j,h}^{'n}\cdot\nabla (\psi_{j,h}^{n+1}-\psi_{j,h}^n+\psi_{j,h}^{n-1}),\psi_{j,h}^{n+1}\big)|+|G_2(t,v_j,w_j,\psi_{j,h}^{n+1})|.\label{baddddn2}
\end{align}
We now turn our attention to finding bounds on the right hand side terms of \eqref{baddddn1}. Applying Cauchy-Schwarz and Young's inequalities on the first five turns provides
\begin{align*}
    (1+\theta)\frac{|\nu-\nu_m|}{2}|\big(\nabla\psi_{j,h}^n,\nabla\phi_{j,h}^{n+1}\big)|&\le(1+\theta)\frac{|\nu-\nu_m|}{4}\big(\|\nabla\psi_{j,h}^n\|^2+\|\nabla\phi_{j,h}^{n+1}\|^2\big),\\
    \theta\frac{|\nu-\nu_m|}{2}|\big(\nabla\psi_{j,h}^{n-1},\nabla\phi_{j,h}^{n+1}\big)|&\le\theta\frac{|\nu-\nu_m|}{4}\big(\|\nabla\psi_{j,h}^{n-1}\|^2+\|\nabla\phi_{j,h}^{n+1}\|^2\big),\\
    (1+\theta)\frac{|\nu-\nu_m|}{2}|\big(\nabla\eta_{j,w}^n,\nabla\phi_{j,h}^{n+1}\big)|&\le\frac{\alpha}{36}\|\nabla\phi_{j,h}^{n+1}\|^2+\frac{9(1+\theta)^2(\nu-\nu_m)^2}{4\alpha}\|\nabla\eta_{j,w}^n\|^2,\\
    \theta\frac{|\nu-\nu_m|}{2}|\big(\nabla\eta_{j,w}^{n-1},\nabla\phi_{j,h}^{n+1}\big)|&\le \frac{\alpha}{36}\|\nabla\phi_{j,h}^{n+1}\|^2+\frac{9\theta^2(\nu-\nu_m)^2}{4\alpha}\|\nabla\eta_{j,w}^{n-1}\|^2,\\
    \frac{\nu+\nu_m}{2}|\big(\nabla \eta_{j,v}^{n+1},\nabla \phi_{j,h}^{n+1}\big)|&\le\frac{\alpha}{36}\|\nabla\phi_{j,h}^{n+1}\|^2+\frac{9(\nu+\nu_m)^2}{4\alpha}\|\nabla \eta_{j,v}^{n+1}\|^2.
\end{align*}

Applying Young’s inequalities with \eqref{nonlinearbound} on the first three nonlinear terms yields
\begin{align*}
    |\big((2\eta_{j,w}^n-\eta_{j,w}^{n-1})\cdot\nabla v_j(t^{n+1}),\phi_{j,h}^{n+1}\big)|&\le C\|\nabla(2\eta_{j,w}^n-\eta_{j,w}^{n-1})\|\|\nabla v_j(t^{n+1})\|\|\nabla\phi_{j,h}^{n+1}\|\\&\le \frac{\alpha}{36}\|\nabla\phi_{j,h}^{n+1}\|^2+\frac{C}{\alpha}\|\nabla(2\eta_{j,w}^n-\eta_{j,w}^{n-1})\|^2\|\nabla v_j(t^{n+1})\|^2,\\|\big((2w_{j,h}^n-w_{j,h}^{n-1})\cdot\nabla\eta_{j,v}^{n+1},\phi_{j,h}^{n+1}\big)|&\le C\|\nabla  (2w_{j,h}^n-w_{j,h}^{n-1})\|\|\nabla\eta_{j,v}^{n+1}\|\|\nabla  \phi_{j,h}^{n+1}\|\\
    &\le \frac{\alpha}{36}\|\nabla\phi_{j,h}^{n+1}\|^2+\frac{C}{\alpha}\|\nabla  (2w_{j,h}^n-w_{j,h}^{n-1})\|^2\|\nabla\eta_{j,v}^{n+1}\|^2,\\
    |\big(w_{j,h}^{'n}\cdot\nabla(\eta_{j,v}^{n+1}-\eta_{j,v}^n+\eta_{j,v}^{n-1}),\phi_{j,h}^{n+1}\big)|&\le C\|\nabla w_{j,h}^{'n}\|\|\nabla(\eta_{j,v}^{n+1}-\eta_{j,v}^n+\eta_{j,v}^{n-1})\|\|\nabla\phi_{j,h}^{n+1}\|\\
    &\le \frac{\alpha}{36}\|\nabla\phi_{j,h}^{n+1}\|^2+\frac{C}{\alpha}\|\nabla w_{j,h}^{'n}\|^2\|\nabla(\eta_{j,v}^{n+1}-\eta_{j,v}^n+\eta_{j,v}^{n-1})\|^2.
\end{align*}
For the fourth nonlinear term, we use H\"older’s inequality, Sobolev embedding theorems,
Poincare’s and Young’s inequalities to reveal
\begin{align*}
    |\big((2\psi_{j,h}^n-\psi_{j,h}^{n-1})\cdot\nabla v_j(t^{n+1}),\phi_{j,h}^{n+1}\big)|&\le C\|2\psi_{j,h}^n-\psi_{j,h}^{n-1}\|\|\nabla v_j(t^{n+1})\|_{L^6}\| \phi_{j,h}^{n+1}\|_{L^3}\\&\le C \|2\psi_{j,h}^n-\psi_{j,h}^{n-1}\|\| v_j(t^{n+1})\|_{H^2}\| \phi_{j,h}^{n+1}\|^{1/2}\|\nabla \phi_{j,h}^{n+1}\|^{1/2}\\
    &\le  C \|2\psi_{j,h}^n-\psi_{j,h}^{n-1}\|\| v_j(t^{n+1})\|_{H^2}\|\nabla \phi_{j,h}^{n+1}\|\\&\le \frac{\alpha}{36}\|\nabla\phi_{j,h}^{n+1}\|^2+\frac{C}{\alpha}\|v_j(t^{n+1})\|_{H^2}^2\|2\psi_{j,h}^n-\psi_{j,h}^{n-1}\|^2.
\end{align*}
Applying inverse and Young’s inequalities with \eqref{nonlinearbound} on the fifth nonlinear term yields
\begin{align*}
    |\big(w_{j,h}^{'n}\cdot\nabla (\phi_{j,h}^{n+1}-\phi_{j,h}^n+\phi_{j,h}^{n-1}),\phi_{j,h}^{n+1}\big)|&\le C\|\nabla w_{j,h}^{'n}\|\|\nabla (\phi_{j,h}^{n+1}-\phi_{j,h}^n+\phi_{j,h}^{n-1})\|\|\nabla\phi_{j,h}^{n+1}\|
    \\&\le\frac{\alpha}{36}\|\nabla\phi_{j,h}^{n+1}\|^2+\frac{C}{\alpha h^2}\|\nabla w_{j,h}^{'n}\|^2\| \phi_{j,h}^{n+1}-\phi_{j,h}^n+\phi_{j,h}^{n-1}\|^2.
\end{align*}
Using Taylor's series, Cauchy-Schwarz, Poincare's and Young's inequalities the last term is evaluated as
\begin{align*}
    |G_1(t,v_j,w_j,\phi_{j,h}^{n+1})|\le &\frac{\alpha}{36}\|\nabla\phi_{j,h}^{n+1}\|^2+(\Delta t)^2\frac{9(\nu-\nu_m)^2(1-\theta)^2}{\alpha}\|\nabla  w_t(s^{***})\|^2\\&+(\Delta t)^4\frac{C}{\alpha}\bigg\{  \|\nabla w_{tt}(s^*)\|^2\|\nabla v_j(t^{n+1})\|^2+\|\nabla w_{j,h}^{'n}\|^2\|\nabla v_{tt}(s^{**})\|^2+\|v_{ttt}(s^{****})\|^2\bigg\},
\end{align*}
with $s^{*}, s^{**}, s^{***}, s^{****}\in[t^{n-1},t^{n+1}]$. Using these estimates in \eqref{baddddn1} and reducing produces
\begin{align}
\frac{1}{4\Delta t}\big(\|\phi_{j,h}^{n+1}\|^2-\|\phi_{j,h}^{n}\|^2+\|2\phi_{j,h}^{n+1}-\phi_{j,h}^{n}\|^2-\|2\phi_{j,h}^{n}-\phi_{j,h}^{n-1}\|^2\big)\nonumber\\\bigg(\frac{1}{4\Delta t}-\frac{C}{\alpha h^2}\|\nabla w_{j,h}^{'n}\|^2\bigg)\|\phi_{j,h}^{n+1}-2\phi_{j,h}^{n}+\phi_{j,h}^{n-1}\|^2+\frac{\nu+\nu_m}{4}\|\nabla \phi_{j,h}^{n+1}\|^2\nonumber\\\le(1+\theta)\frac{|\nu-\nu_m|}{4}\|\nabla\psi_{j,h}^n\|^2+\theta\frac{|\nu-\nu_m|}{4}\|\nabla\psi_{j,h}^{n-1}\|^2+\frac{9(1+\theta)^2(\nu-\nu_m)^2}{4\alpha}\|\nabla\eta_{j,w}^n\|^2\nonumber\\+\frac{9\theta^2(\nu-\nu_m)^2}{4\alpha}\|\nabla\eta_{j,w}^{n-1}\|^2+\frac{9(\nu+\nu_m)^2}{4\alpha}\|\nabla \eta_{j,v}^{n+1}\|^2+\frac{C}{\alpha}\|\nabla(2\eta_{j,w}^n-\eta_{j,w}^{n-1})\|^2\|\nabla v_j(t^{n+1})\|^2\nonumber\\+\frac{C}{\alpha}\|\nabla  (2w_{j,h}^n-w_{j,h}^{n-1})\|^2\|\nabla\eta_{j,v}^{n+1}\|^2+\frac{C}{\alpha}\|\nabla w_{j,h}^{'n}\|^2\|\nabla(\eta_{j,v}^{n+1}-\eta_{j,v}^n+\eta_{j,v}^{n-1})\|^2\nonumber\\+\frac{C}{\alpha}\|v_j(t^{n+1})\|_{H^2}^2\|2\psi_{j,h}^n-\psi_{j,h}^{n-1}\|^2+(\Delta t)^2\frac{9(\nu-\nu_m)^2(1-\theta)^2}{\alpha}\|\nabla  w_t(s^{***})\|^2\nonumber\\+(\Delta t)^4\frac{C}{\alpha}\bigg\{  \|\nabla w_{tt}(s^*)\|^2\|\nabla v_j(t^{n+1})\|^2+\|\nabla w_{j,h}^{'n}\|^2\|\nabla v_{tt}(s^{**})\|^2+\|v_{ttt}(s^{****})\|^2\bigg\}.\label{upbd1}
\end{align}
Apply similar techniques to \eqref{baddddn2}, we get
\begin{align}
\frac{1}{4\Delta t}\big(\|\psi_{j,h}^{n+1}\|^2-\|\psi_{j,h}^{n}\|^2+\|2\psi_{j,h}^{n+1}-\psi_{j,h}^{n}\|^2-\|2\psi_{j,h}^{n}-\psi_{j,h}^{n-1}\|^2\big)\nonumber\\+\bigg(\frac{1}{4\Delta t}-\frac{C}{\alpha h^2}\|\nabla v_{j,h}^{'n}\|^2\bigg)\|\psi_{j,h}^{n+1}-2\psi_{j,h}^{n}+\psi_{j,h}^{n-1}\|^2+\frac{\nu+\nu_m}{4}\|\nabla \psi_{j,h}^{n+1}\|^2\nonumber\\\le(1+\theta)\frac{|\nu-\nu_m|}{4}\|\nabla\phi_{j,h}^n\|^2+\theta\frac{|\nu-\nu_m|}{4}\|\nabla\phi_{j,h}^{n-1}\|^2+\frac{9(1+\theta)^2(\nu-\nu_m)^2}{4\alpha}\|\nabla\eta_{j,v}^n\|^2\nonumber\\+\frac{9\theta^2(\nu-\nu_m)^2}{4\alpha}\|\nabla\eta_{j,v}^{n-1}\|^2+\frac{9(\nu+\nu_m)^2}{4\alpha}\|\nabla \eta_{j,w}^{n+1}\|^2+\frac{C}{\alpha}\|\nabla(2\eta_{j,v}^n-\eta_{j,v}^{n-1})\|^2\|\nabla w_j(t^{n+1})\|^2\nonumber\\+\frac{C}{\alpha}\|\nabla  (2v_{j,h}^n-v_{j,h}^{n-1})\|^2\|\nabla\eta_{j,w}^{n+1}\|^2+\frac{C}{\alpha}\|\nabla v_{j,h}^{'n}\|^2\|\nabla(\eta_{j,w}^{n+1}-\eta_{j,w}^n+\eta_{j,w}^{n-1})\|^2\nonumber\\+\frac{C}{\alpha}\|w_j(t^{n+1})\|_{H^2}^2\|2\phi_{j,h}^n-\phi_{j,h}^{n-1}\|^2+(\Delta t)^2\frac{9(\nu-\nu_m)^2(1-\theta)^2}{\alpha}\|\nabla  v_t(s^{***})\|^2\nonumber\\+(\Delta t)^4\frac{C}{\alpha}\bigg\{  \|\nabla v_{tt}(t^*)\|^2\|\nabla w_j(t^{n+1})\|^2+\|\nabla v_{j,h}^{'n}\|^2\|\nabla w_{tt}(t^{**})\|^2+\|w_{ttt}(t^{****})\|^2\bigg\},\label{upbd2}
\end{align}
with $t^{*}, t^{**}, t^{***}, t^{****}\in[t^{n-1},t^{n+1}]$. Now adds \eqref{upbd1} and \eqref{upbd2}, multiply by $4\Delta t$, assume $$\Delta t\le\frac{\alpha h^2}{C\max\limits_{1\le j\le J}\{\|\nabla v_{j,h}^{'n}\|,\|\nabla w_{j,h}^{'n}\|\}},$$ drops non-negative terms, use stability and
regularity assumptions, $\|\phi_{j,h}^0\|=\|\psi_{j,h}^0\|=\|\phi_{j,h}^1\|=\|\psi^1_{j,h}\|=0$, $\Delta tM=T$, and sum over the
time steps to find
\begin{align}
\|\phi_{j,h}^M&\|^2+\|2\phi_{j,h}^M-\phi_{j,h}^{M-1}\|^2+\|\psi_{j,h}^M\|^2+\|2\psi_{j,h}^M-\psi_{j,h}^{M-1}\|^2\nonumber\\&+\alpha\Delta t\sum\limits_{n=2}^{M}\big(\|\nabla\phi_{j,h}^n\|^2+\|\nabla\psi_{j,h}^n\|^2\big)\le \frac{9(1+\theta)^2(\nu-\nu_m)^2}{\alpha}\Delta  t\sum\limits_{n=1}^{M-1}(\|\nabla\eta_{j,v}^{n}\|^2+\|\nabla\eta_{j,w}^{n}\|^2)\nonumber\\&+\frac{9\theta^2(\nu-\nu_m)^2}{\alpha}\Delta t\sum\limits_{n=1}^{M-1}(\|\nabla\eta_{j,v}^{n-1}\|^2+\|\nabla\eta_{j,w}^{n-1}\|^2)+\frac{9(\nu+\nu_m)^2}{\alpha}\Delta t\sum\limits_{n=1}^{M-1}(\|\nabla\eta_{j,v}^{n+1}\|^2+\|\nabla\eta_{j,w}^{n+1}\|^2)\nonumber\\&+\frac{C}{\alpha}\Delta t\sum\limits_{n=1}^{M-1}\bigg\{\|\nabla(2\eta_{j,w}^n-\eta_{j,w}^{n-1})\|^2\|\nabla v_j(t^{n+1})\|^2+\|\nabla(2\eta_{j,v}^n-\eta_{j,v}^{n-1})\|^2\|\nabla w_j(t^{n+1})\|^2\bigg\}\nonumber\\&+\frac{C}{\alpha}\Delta t\sum\limits_{n=1}^{M-1}\bigg\{\|\nabla(2w_{j,h}^n-w_{j,h}^{n-1})\|^2\|\nabla\eta_{j,v}^{n+1}\|^2+\|\nabla(2v_{j,h}^n-v_{j,h}^{n-1})\|^2\|\nabla\eta_{j,w}^{n+1}\|^2\bigg\}\nonumber\\&+\frac{C}{\alpha}\Delta t\sum\limits_{n=1}^{M-1}\bigg\{\|\nabla w_{j,h}^{'n}\|^2\|\nabla(\eta_{j,v}^{n+1}-\eta_{j,v}^n+\eta_{j,v}^{n-1})\|^2+\|\nabla v_{j,h}^{'n}\|^2\|\nabla(\eta_{j,w}^{n+1}-\eta_{j,w}^n+\eta_{j,w}^{n-1})\|^2\bigg\}\nonumber\\&+\frac{C}{\alpha}\Delta t\sum\limits_{n=1}^{M-1}\bigg\{\|v_j(t^{n+1})\|_{L^\infty(0,T;H^2(\Omega)}^2\|2\psi_{j,h}^n-\psi_{j,h}^{n-1}\|^2+\|w_j(t^{n+1})\|_{H^2}^2\|2\phi_{j,h}^n-\phi_{j,h}^{n-1}\|^2\bigg\}\nonumber\\&+C\left((\Delta t)^4+(\nu-\nu_m)^2(1-\theta)^2(\Delta t)^2\right).
\end{align}
Applying the discrete Gronwall lemma, we have
\begin{align}
\|\phi_{j,h}^M\|^2+&\|2\phi_{j,h}^M-\phi_{j,h}^{M-1}\|^2+\|\psi_{j,h}^M\|^2+\|2\psi_{j,h}^M-\psi_{j,h}^{M-1}\|^2+\alpha\Delta t\sum\limits_{n=2}^{M}\left(\|\nabla\phi_{j,h}^n\|^2+\|\nabla\psi_{j,h}^n\|^2\right)\nonumber\\&\le C \left(h^{2k}+(\Delta t)^4+(\nu-\nu_m)^2(1-\theta)^2\Delta t^2\right).
\end{align}
Using the triangular inequality allows us to write
\begin{align}
&\|e_{j,v}^M\|^2+\|e_{j,w}^M\|^2+\alpha\Delta t\sum\limits_{n=2}^{M}\left(\|\nabla e_{j,v}^n\|^2+\|\nabla e_{j,w}^n\|^2\right)\le 2\bigg(\|\phi_{j,h}^M\|^2+\|\psi_{j,h}^M\|^2\nonumber\\&+\alpha\Delta t\sum\limits_{n=2}^{M}\left(\|\nabla\phi_{j,h}^n\|^2+\|\nabla\psi_{j,h}^n\|^2\right)+\|\eta_{j,v}^M\|^2+\|\eta_{j,w}^M\|^2+\alpha\Delta t\sum\limits_{n=2}^{M}\left(\|\nabla\eta_{j,v}^n\|^2+\|\nabla\eta_{j,w}^n\|^2\right)\bigg)\nonumber\\&\le C \left(h^{2k}+(\Delta t)^4+(\nu-\nu_m)^2(1-\theta)^2\Delta t^2\right).
\end{align}

Now summing over $j$ and using the triangular inequality completes the proof.
\end{proof}
\section{\large Numerical experiments}
To test the proposed Algorithm \ref{Algn1} and theory, in this section we present results of numerical experiments. In order to compute the first timestep solutions $v_j^1$, and $w_j^1$, we use a first-order ensemble backward-Euler scheme \textcolor{black}{proposed} in \cite{MR17} \textcolor{black}{without the eddy viscosity term and} together with the initial conditions $v_j^0=v_j(0)$ and $w_j^0=w_j(0)$. Thus, for further time evolution, $v_j^0,\hspace{1mm}w_j^0,\hspace{1mm}v_j^1,\hspace{1mm}\text{and}\hspace{1mm}w_j^1$ are used as the two required initial conditions for the Algorithm \ref{Algn1}.

For \textcolor{black}{simulations} of MHD systems, it is considered important to enforce the solenoidal constraint $\nabla\cdot B=0$ in discrete level to the machine precision \cite{HMX17}. This is because the condition is induced by the induction equation \textcolor{black}{for all time if the initial magnetic field is divergence free \cite{mohebujjaman2020scalability}}, which is a precise physical law. Moreover, it has been shown that for MHD flow simulations $\nabla\cdot B\ne 0$ can produce large errors in the solution \cite{BB80}. Thus, the $((P_2)^2,P_1^{disc})$ SV element,  which is pointwise divergence-free on a barycenter refined regular triangular mesh, will be used for the velocity-pressure and magnetic field-magnetic pressure pairs throughout this section. We ran our simulations using the free finite element software Freefem++\cite{H12} with the triangular mesh. We used direct solver UMFPACK for all the simulations.

\subsection{Convergence rate verification} 
To verify the predicted convergence rates of our analysis in Section \ref{ErrorAnalysis}, we begin the experiment with a manufactured analytical solution,
\[
{v}=\left(\begin{array}{c} \cos y+(1+e^t)\sin y \\ \sin x+(1+e^t)\cos x \end{array} \right), \
{w}=\left(\begin{array}{c} \cos y-(1+e^t)\sin y \\ \sin x-(1+e^t)\cos x \end{array} \right), \ p =\sin(x+y)(1+e^t), \ \lambda=0,
\]
on the domain $\Omega = (0,1)^2$.
Next, we create four different true solutions from the above solution introducing a perturbation parameter $\epsilon$ as follows: 
\begin{align*}
    v_j:=\begin{cases} 
(1+(-1)^{j-1}\epsilon)v & 1\le j<3 \\
(1+(-1)^{j-1}2\epsilon)v & 3\le j\le 4
\end{cases},\hspace{2mm}\text{and}\hspace{2mm} w_j:=\begin{cases} 
(1+(-1)^{j-1}\epsilon)w & 1\le j<3 \\
(1+(-1)^{j-1}2\epsilon)w & 3\le j\le 4
\end{cases},
\end{align*}
where $j\in\mathbb{N}$. By construction, the averages $\bar{v_j}=v$ and $\bar{w_j}=w$. Using the above perturbed solutions, we compute right-hand side forcing terms. The Dirichlet boundary conditions are used on the boundary of the unit square.

The Algorithm \ref{Algn1} computes the discrete ensemble average $<v_h^n>$ and $<w_h^n>$, and these will be used to compare to the true ensemble average $<v(t^n)>$ and $<w(t^n)>$, respectively. We notate the ensemble average error as $<e_u>:=<u_h>^n-<u(t^n)>$.

For our choice of element, the theory predicts the $L^2(0, T;H^1(\Omega)^d)$ error to be of $O(h^2+\Delta t^2+(1-\theta)|\nu-\nu_m|\Delta t)$ provided $\Delta t<O(h^2)$. In this experiment, we consider $\nu=0.01$, and $\nu_m=0.001$ and compute the largest possible $\theta=1/9$, subject to the condition \textcolor{red}{in} \eqref{theta_con}. In this case, $\Delta t^2$ dominates over $(1-\theta)|\nu-\nu_m|\Delta t$, and thus the error behaves like $O(h^2+\Delta t^2)$. We consider three different choices $\epsilon=0.001,\text{ }0.01\text{ } \text{and}\text{ }0.1$ for the perturbation parameter.

To observe the temporal convergence, we choose a fixed \textcolor{black}{mesh width} $h=1/64$, end time $T=1$, and \textcolor{black}{compute with varying timestep sizes.} On the other hand, we use a small end time $T=0.001$, a fixed timestep size $\Delta t=T/8$, and \textcolor{black}{compute on successively refined meshes to observe the spatial convergence rate.} 

Tables \ref{tem_con}-\ref{tab4} exhibit errors and convergence rates for the variable $v$ \textcolor{black}{and $w$}, and we observe second order asymptotic temporal convergence rates and optimal spatial convergence rates for all choices of $\epsilon$. Note that we computed the convergence rates for both the variables and for all the choices of $\epsilon$ using both SV and weakly divergence free \textcolor{black}{Taylor Hood (TH)} elements. For this particular problem, we have found the same convergence behavior with both the SV and TH elements, and thus TH element results are omitted.

\begin{table}[!ht]
	\begin{center}
		\small\begin{tabular}{|c|c|c|c|c|c|c|}\hline
			\multicolumn{7}{|c|}{Temporal convergence (fixed $h=1/64$)}\\\hline
			\multicolumn{1}{|c|}{}&\multicolumn{2}{|c|}{$\epsilon=0.001$}
			& \multicolumn{2}{|c|}{$\epsilon=0.01$}& \multicolumn{2}{|c|}{$\epsilon=0.1$} \\ \hline
			$\Delta t$ & $\|<e_v>\|_{2,1}$ & rate   &$\|<e_v>\|_{2,1}$  & rate & $\|<e_v>\|_{2,1}$ & rate  \\ \hline
			$\frac{T}{4}$ & 2.8765e-1 &  & 2.8767e-1 &  & 2.9004e-1& \\ \hline
			$\frac{T}{8}$ & 8.4966e-2 & 1.76 & 8.4974e-2 &1.76  & 8.5986e-2  & 1.75        \\ \hline
			$\frac{T}{16}$& 2.3855e-2 & 1.83 & 2.3860e-2 & 1.83 & 2.4048e-2  & 1.84         \\ \hline
			$\frac{T}{32}$& 6.2895e-3 &1.92  & 6.2899e-3 &1.92  &  6.3445e-3 & 1.92   \\ \hline
			$\frac{T}{64}$&1.5801e-3  & 1.99 & 1.5801e-3 &1.99  &   1.5938e-3&  1.99\\ \hline
		\end{tabular}
	\end{center}
	\caption{\footnotesize Errors and  convergence rates for $v$ with $\theta=1/9$, $\nu=0.01$, and $\nu_m=0.001$.}\label{tem_con}
\end{table}
\begin{table}[!ht]
	\begin{center}
		\small\begin{tabular}{|c|c|c|c|c|c|c|}\hline
			\multicolumn{7}{|c|}{Spatial convergence (fixed $T=0.001$, $\Delta t=T/8$)}\\\hline
			\multicolumn{1}{|c|}{}&\multicolumn{2}{|c|}{$\epsilon=0.001$}
			& \multicolumn{2}{|c|}{$\epsilon=0.01$}& \multicolumn{2}{|c|}{$\epsilon=0.1$} \\ \hline
			$h$ & $\|<e_v>\|_{2,1}$ & rate   &$\|<e_v>\|_{2,1}$  & rate & $\|<e_v>\|_{2,1}$ & rate  \\ \hline
			 $\frac{1}{4}$ & 1.2071e-4 &  & 1.2071e-4 &  & 1.2071e-4& \\ \hline
			 $\frac{1}{8}$ &3.0380e-5  & 1.99 & 3.0380e-5  & 1.99 & 3.0382e-5   &  1.99       \\ \hline
			 $\frac{1}{16}$& 7.6186e-6 & 2.00 & 7.6186e-6 & 2.00 & 7.6197e-6 & 2.00         \\ \hline
			 $\frac{1}{32}$& 1.9144e-6 & 1.99 & 1.9144e-6 & 1.99 & 1.9151e-6  & 1.99   \\ \hline
			 $\frac{1}{64}$& 4.8147e-7 & 1.99 &4.8147e-7 & 1.99 & 4.8180e-7  & 1.99   \\ \hline
		\end{tabular}
	\end{center}
	\caption{\footnotesize Errors and  convergence rates for $v$ with $\theta=1/9$, $\nu=0.01$, and $\nu_m=0.001$.}\label{spac_con}
\end{table}

\begin{table}[!ht]
	\begin{center}
		\small\begin{tabular}{|c|c|c|c|c|c|c|}\hline
			\multicolumn{7}{|c|}{Spatial convergence (fixed $T=0.001$, $\Delta t=T/8$)}\\\hline
			\multicolumn{1}{|c|}{}&\multicolumn{2}{|c|}{$\epsilon=0.001$}
			& \multicolumn{2}{|c|}{$\epsilon=0.01$}& \multicolumn{2}{|c|}{$\epsilon=0.1$} \\ \hline
			$h$ & $\|<e_w>\|_{2,1}$ & rate   &$\|<e_w>\|_{2,1}$  & rate & $\|<e_w>\|_{2,1}$ & rate  \\ \hline
			   $\frac{1}{4}$ & 2.3107e-4 &  & 2.3107e-4 &  & 2.3108e-4& \\ \hline
			 $\frac{1}{8}$ & 5.7827e-5 & 2.00 & 5.7827e-5 & 2.00  & 5.7832e-5  & 2.00 \\ \hline
			 $\frac{1}{16}$& 1.4539e-5 & 1.99 & 1.4539e-5 & 1.99 & 1.4544e-5 & 1.99\\ \hline
			 $\frac{1}{32}$& 3.6966e-6 & 1.98 &3.6966e-6  & 1.98 & 3.7008e-6  & 1.97\\ \hline
			 $\frac{1}{64}$& 9.4949e-7 & 1.96 &9.4951e-7  & 1.96 &9.5174e-7  & 1.96 \\ \hline
		\end{tabular}
	\end{center}
	\caption{\footnotesize Errors and  convergence rates for $w$ with $\theta=1/9$, $\nu=0.01$, and $\nu_m=0.001$.}
\end{table}

\begin{table}[!ht]
	\begin{center}
		\small\begin{tabular}{|c|c|c|c|c|c|c|}\hline
			\multicolumn{7}{|c|}{Temporal convergence (fixed $h=1/64$)}\\\hline
			\multicolumn{1}{|c|}{}&\multicolumn{2}{|c|}{$\epsilon=0.001$}
			& \multicolumn{2}{|c|}{$\epsilon=0.01$}& \multicolumn{2}{|c|}{$\epsilon=0.1$} \\ \hline
			$\Delta t$ & $\|<e_w>\|_{2,1}$ & rate   &$\|<e_w>\|_{2,1}$  & rate & $\|<e_w>\|_{2,1}$ & rate  \\ \hline
			$\frac{T}{4}$ & 2.4694e-1 &  & 2.4694e-1 &  & 2.4787e-1&\\ \hline
			$\frac{T}{8}$ & 7.7109e-1 & 1.68 & 7.7111e-1 & 1.76 &  7.7527e-2& 1.68\\ \hline
			$\frac{T}{16}$& 2.2285e-2 &1.79  & 2.2286e-2 & 1.83 &  2.2455e-2 & 1.79\\ \hline
			$\frac{T}{32}$& 6.0150e-3 &1.89  & 6.0151e-3 &1.92  &  6.0531e-3 & 1.89\\ \hline
			$\frac{T}{64}$& 1.5350e-3 & 1.97 & 1.5350e-3 & 1.99 &  1.5440e-3& 1.97\\ \hline
		\end{tabular}
	\end{center}
	\caption{\footnotesize Errors and  convergence rates for $w$ with $\theta=1/9$, $\nu=0.01$, and $\nu_m=0.001$.}\label{tab4}
\end{table}

\subsection{\small MHD channel flow over a step}

Next, we consider a domain which is a $30\times 10$ rectangular channel with a $1\times 1$ step five units away from the inlet into the channel. No slip boundary condition is prescribed for the velocity and $B=<0,1>^T$  is enforced for the magnetic field on the walls. At the inflow, we set $u=<y(10-y)/25, 0>^T$ and $B=<0,1>^T$, and the outflow condition uses a channel of extension 10 units, and at the end of the extension, we set outflow velocity and magnetic field equal to the inflow. 

An ensemble of four different solutions corresponding to the perturbed initial conditions\\ $u_j(0):=\begin{cases} 
	(1+(-1)^{j-1}\epsilon)u_0 & 1\le j<3 \\
	(1+(-1)^{j-1}2\epsilon)u_0 & 3\le j\le 4
\end{cases}$ and $B_j(0):=\begin{cases} 
	(1+(-1)^{j-1}\epsilon)B_0 & 1\le j<3 \\
	(1+(-1)^{j-1}2\epsilon)B_0 & 3\le j\le 4
\end{cases}$ where, $j\in\mathbb{N}$, $u_0:=<y(10-y)/25, 0>^T$ and $B_0:=<0,1>^T$,  and a similar way perturbed inflow and outflow are considered. A \textcolor{black}{triangular unstructured} mesh of the domain that provides a total of \textcolor{black}{3316922} degrees of freedom (dof) is considered, where velocity $\text{dof}=1473898$, magnetic field $\text{dof}=1473898$, pressure $\text{dof}=184563$, \textcolor{black}{and magnetic pressure $\text{dof}=184563$}.  The simulations of the Algorithm \ref{Algn1} are done with various values of $\epsilon$ until $T=40$, with $s=0.001$, $\nu=0.001$, $\nu_m=0.01$, $\theta=1/9$, and timestep size $\Delta t=1$. \textcolor{black}{For the viscosity and magnetic diffusivity pair, we compute the largest possible $\theta$ so that the condition in \eqref{theta_con} holds.} Velocity and magnetic fields ensemble average solutions for varying $\epsilon$ and parameters are plotted in Figures \ref{vel_ens_sol}-\ref{mag_ens_sol}, and compare them to the usual MHD simulation (which is the $\epsilon=0$ case). \textcolor{black}{We observe that the ensemble average solutions appear to converge to the unperturbed solution as $\epsilon\rightarrow 0$, which is expected from our theory.}
\textcolor{black}{Though, in our analysis, a timestep restriction $\Delta t<O(h^2)$ appears due to the use of inverse inequality, in this numerical experiment, we could choose a larger timestep size and ran the simulation successfully for long time.}

\begin{figure}[h!]
\begin{center} \vspace{-35mm}
            \includegraphics[width = 1\textwidth, height=0.5\textwidth,viewport=0 0 1400 890, clip]{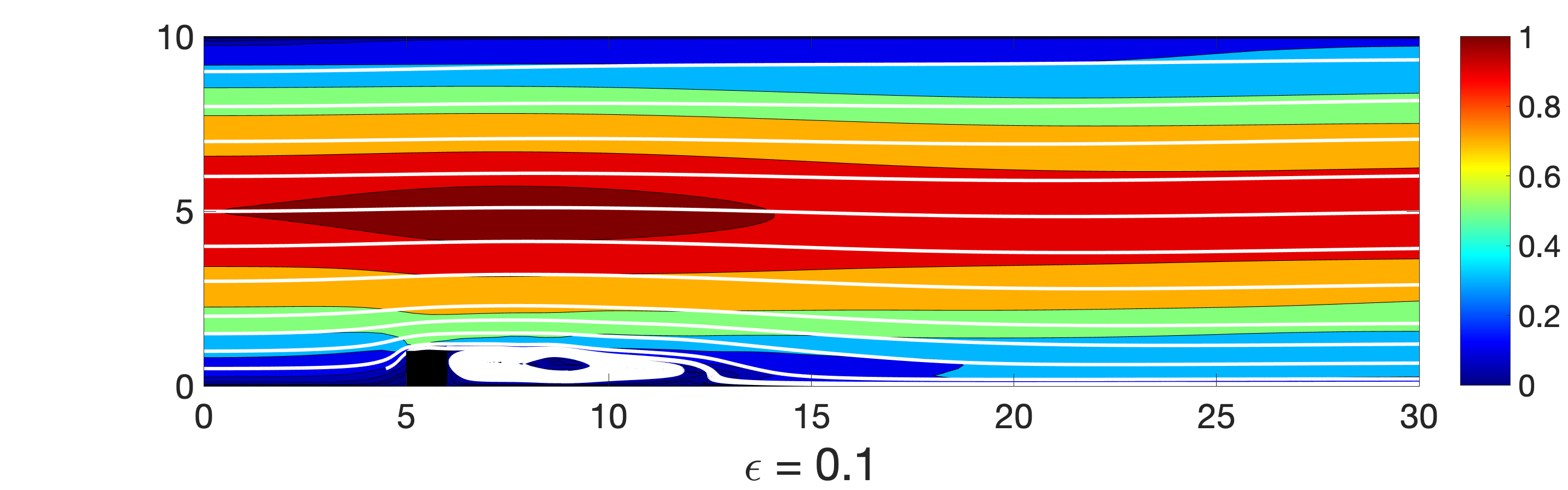}\\\vspace{-37mm}
            \includegraphics[width = 1\textwidth, height=0.5\textwidth,viewport=0 0 1400 890, clip]{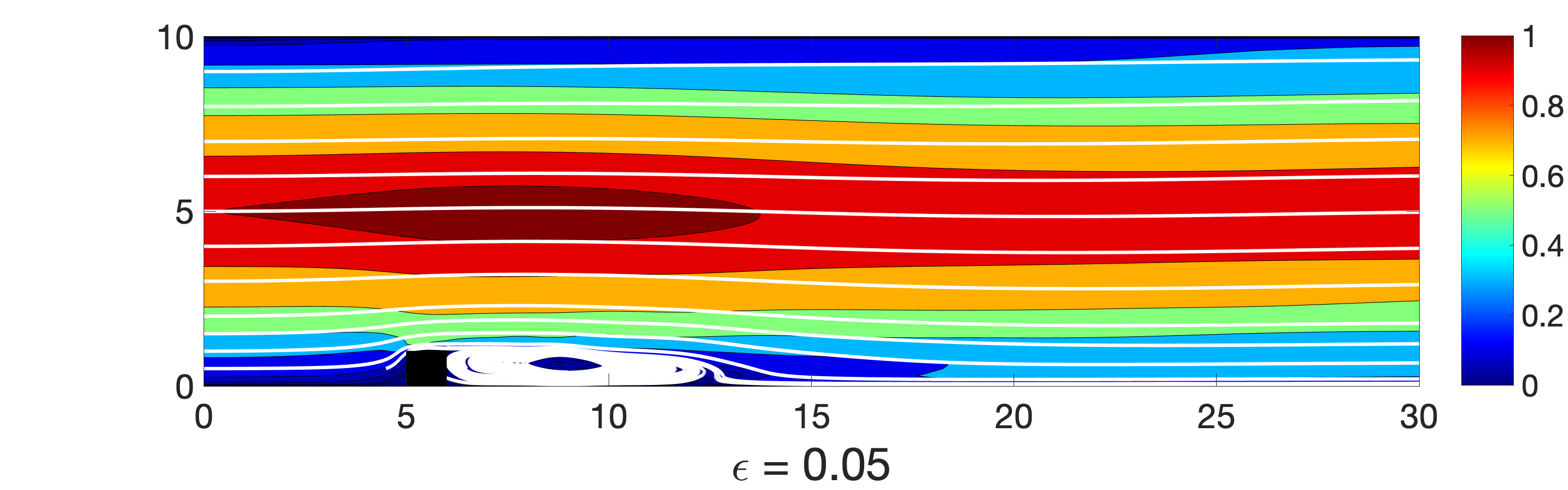}\\\vspace{-37mm}
            \includegraphics[width = 1\textwidth, height=0.5\textwidth,viewport=0 0 1400 890, clip]{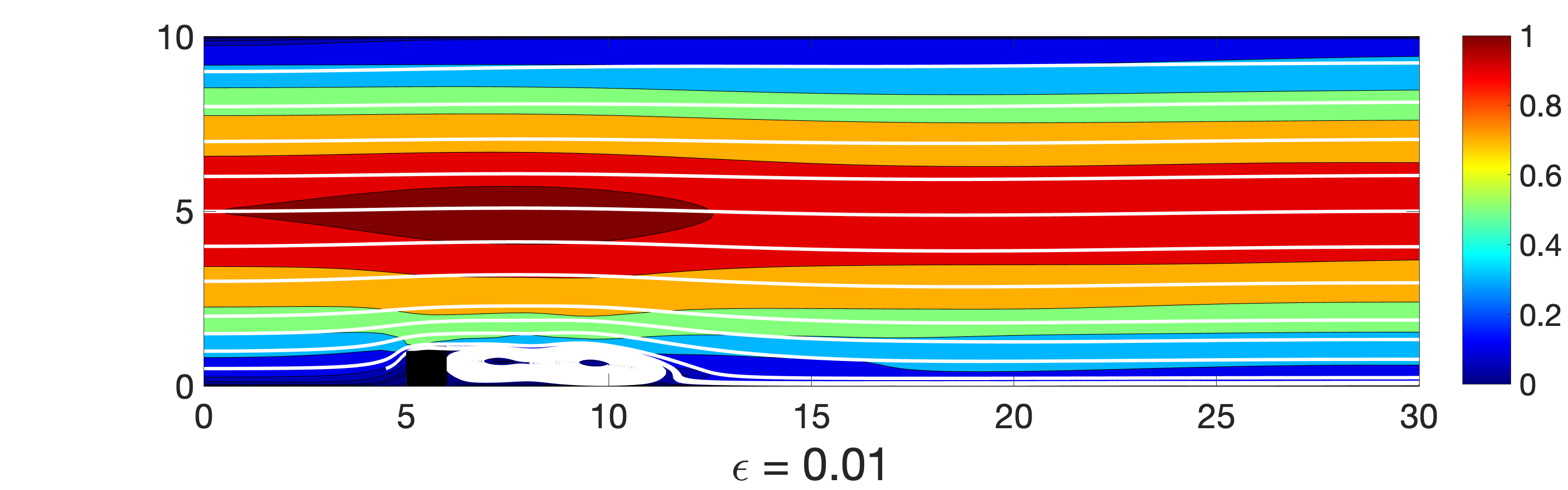}\\\vspace{-37mm}
            \includegraphics[width = 1\textwidth, height=0.5\textwidth,viewport=0 0 1400 890, clip]{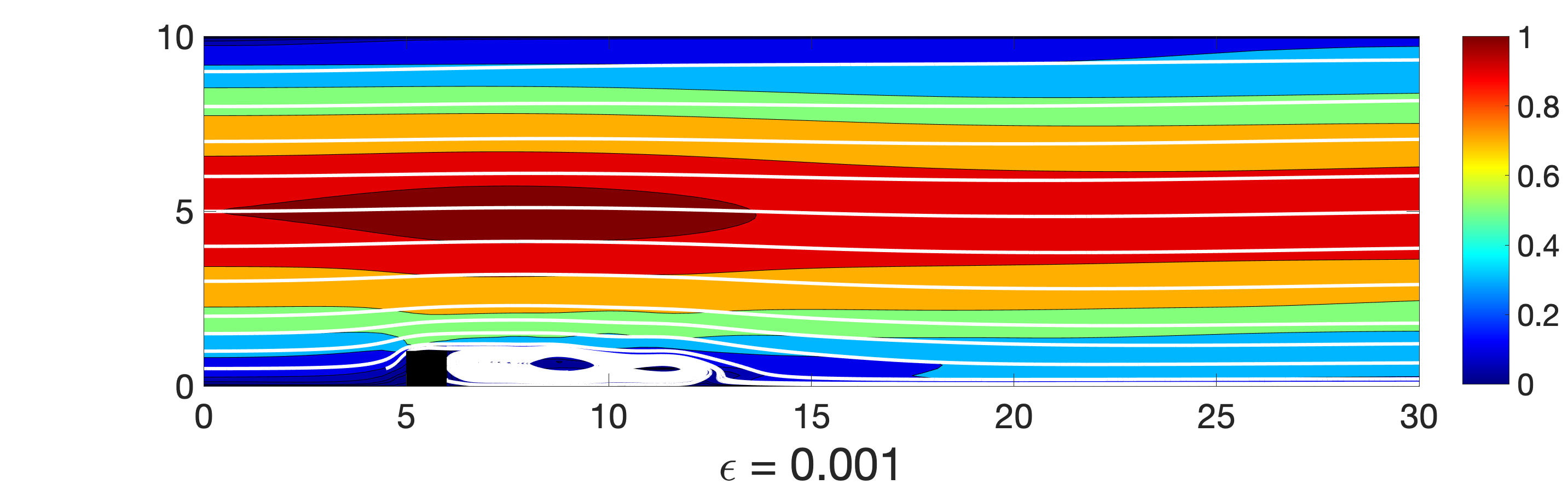}\\\vspace{-37mm}
            \includegraphics[width = 1\textwidth, height=0.5\textwidth,viewport=0 0 1400 890, clip]{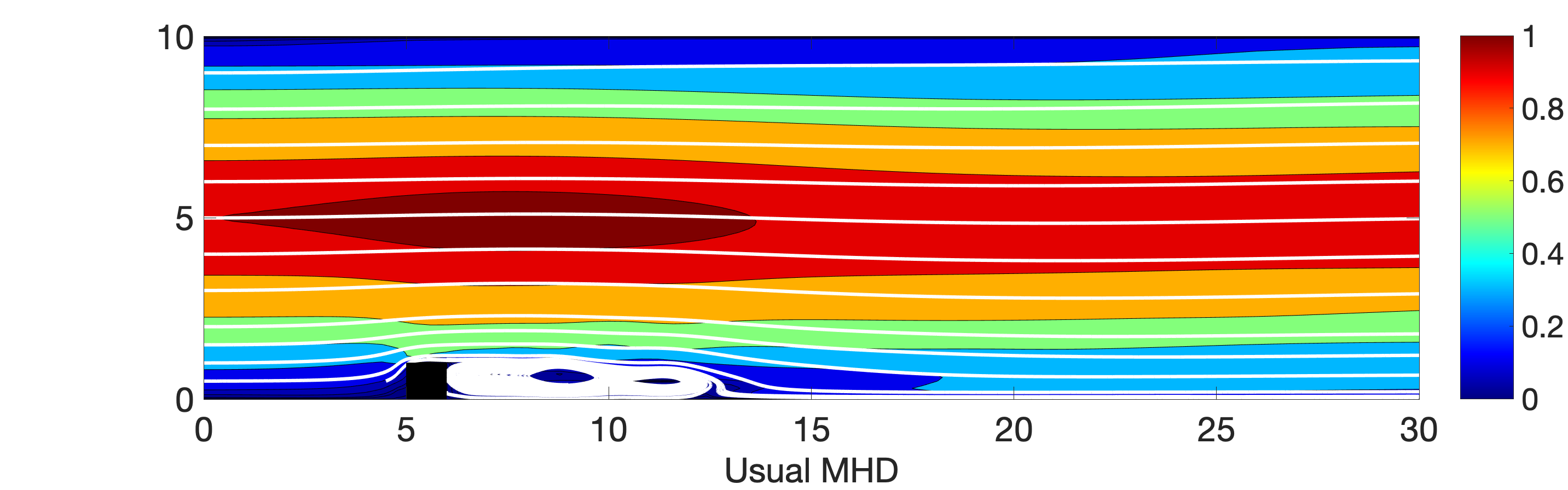}
	 	 	\caption{The velocity ensemble solution (shown as streamlines over speed contour) at $T=40$  for MHD channel flow over a step with $\Delta t=1$, $s=0.001$, $\nu=0.001$, $\nu_m=0.01$, $\theta=1/9$, velocity $\text{dof}=1473898$, and pressure $\text{dof}=184563$.}\label{vel_ens_sol}
\end{center}
\end{figure}

\begin{figure}[h!]
\begin{center} \vspace{-35mm}
            \includegraphics[width = 1\textwidth, height=0.5\textwidth,viewport=0 0 1400 890, clip]{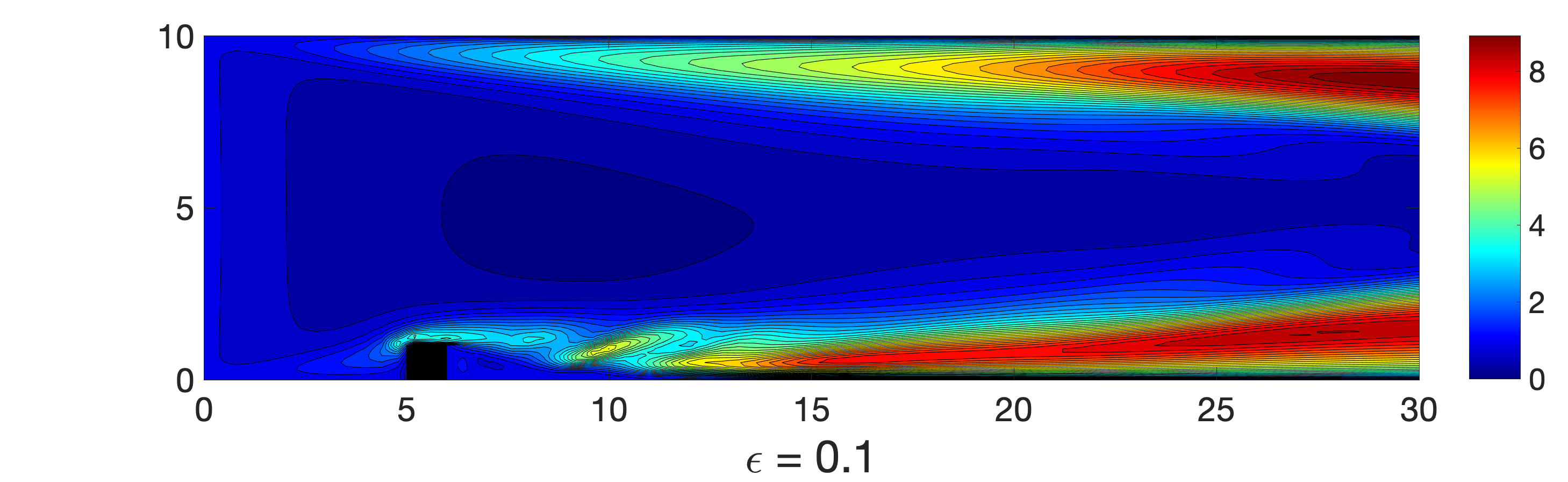}\\\vspace{-37mm}
            \includegraphics[width = 1\textwidth, height=0.5\textwidth,viewport=0 0 1400 890, clip]{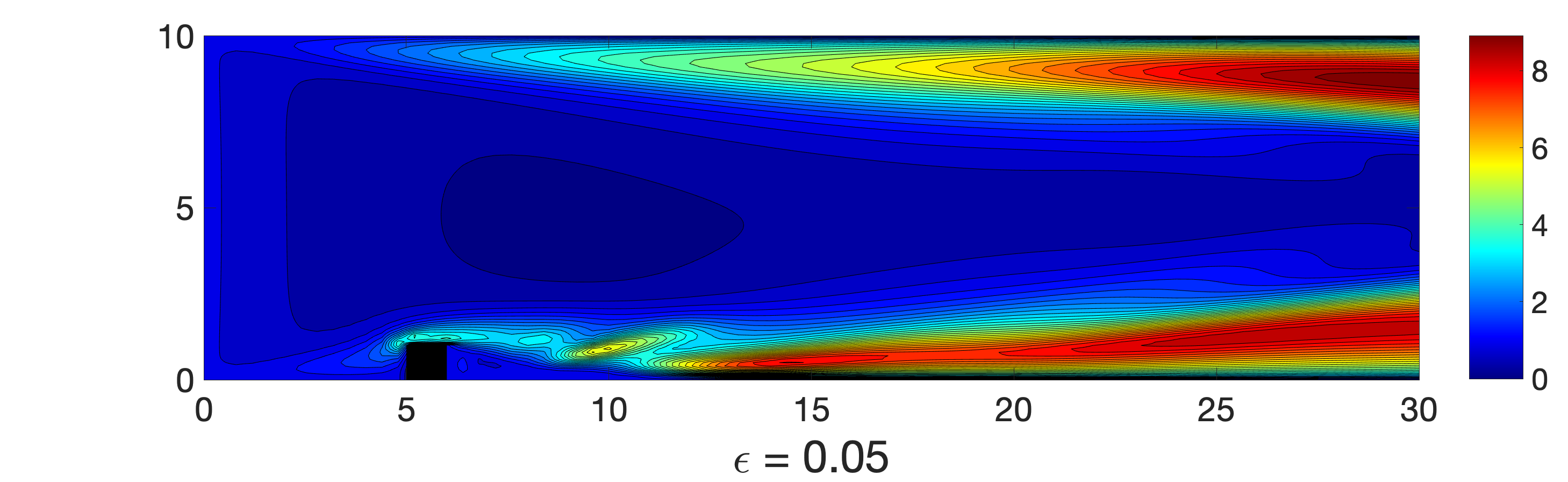}\\\vspace{-37mm}
            \includegraphics[width = 1\textwidth, height=0.5\textwidth,viewport=0 0 1400 890, clip]{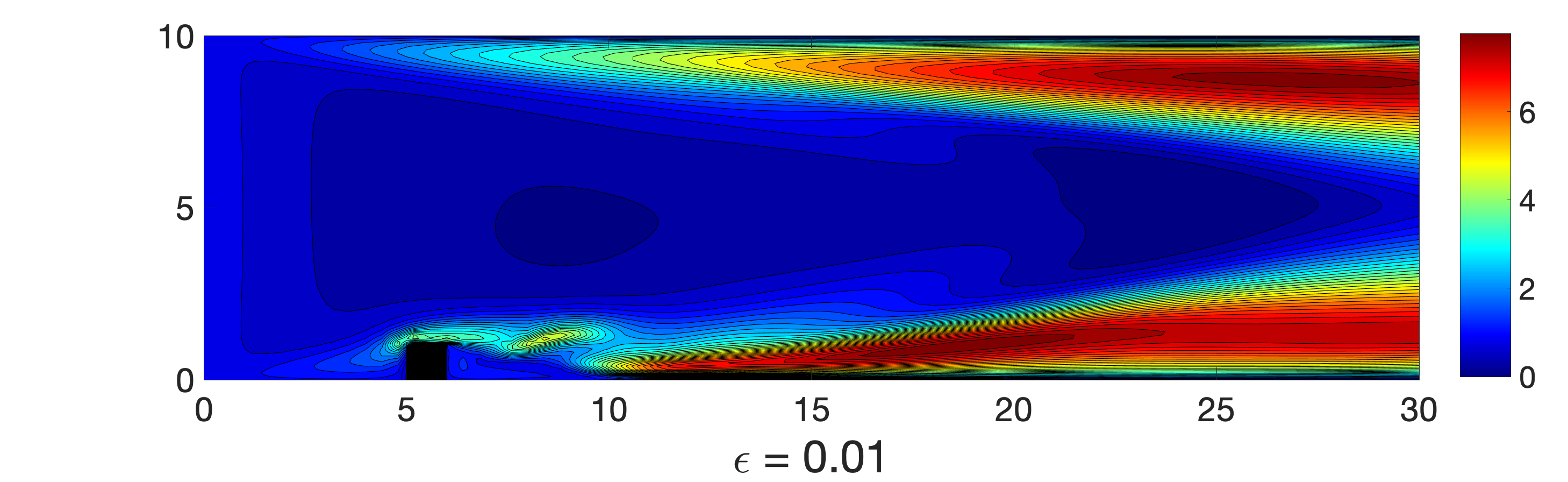}\\\vspace{-37mm}
            \includegraphics[width = 1\textwidth, height=0.5\textwidth,viewport=0 0 1400 890, clip]{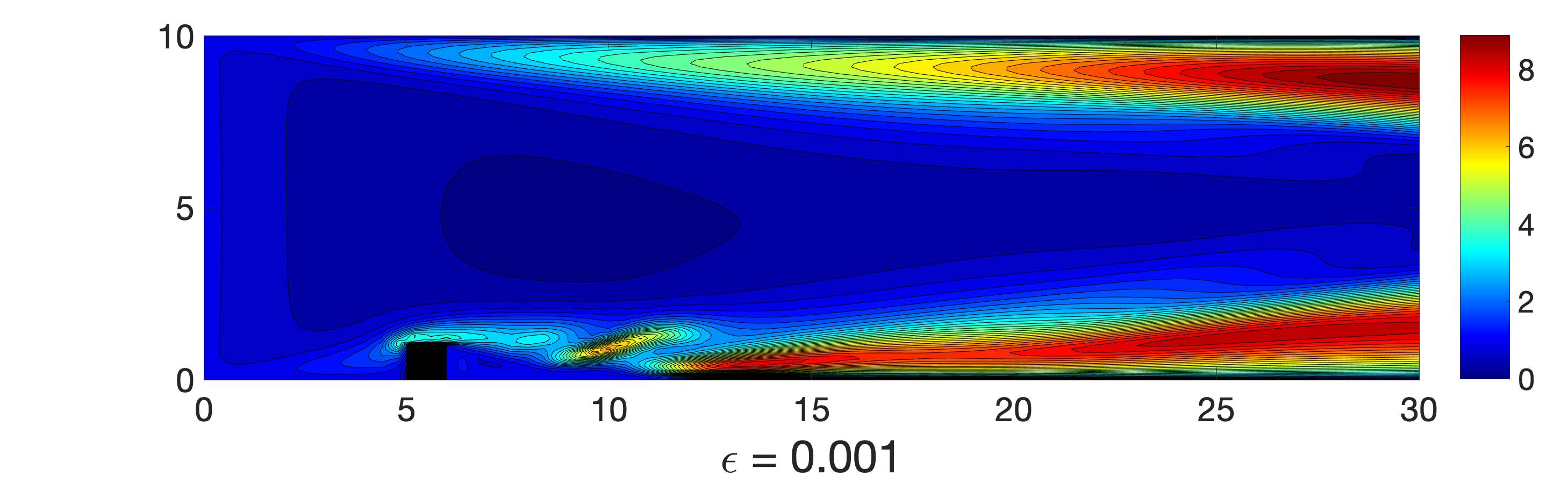}\\\vspace{-37mm}
            \includegraphics[width = 1\textwidth, height=.5\textwidth,viewport=0 0 1400 890, clip]{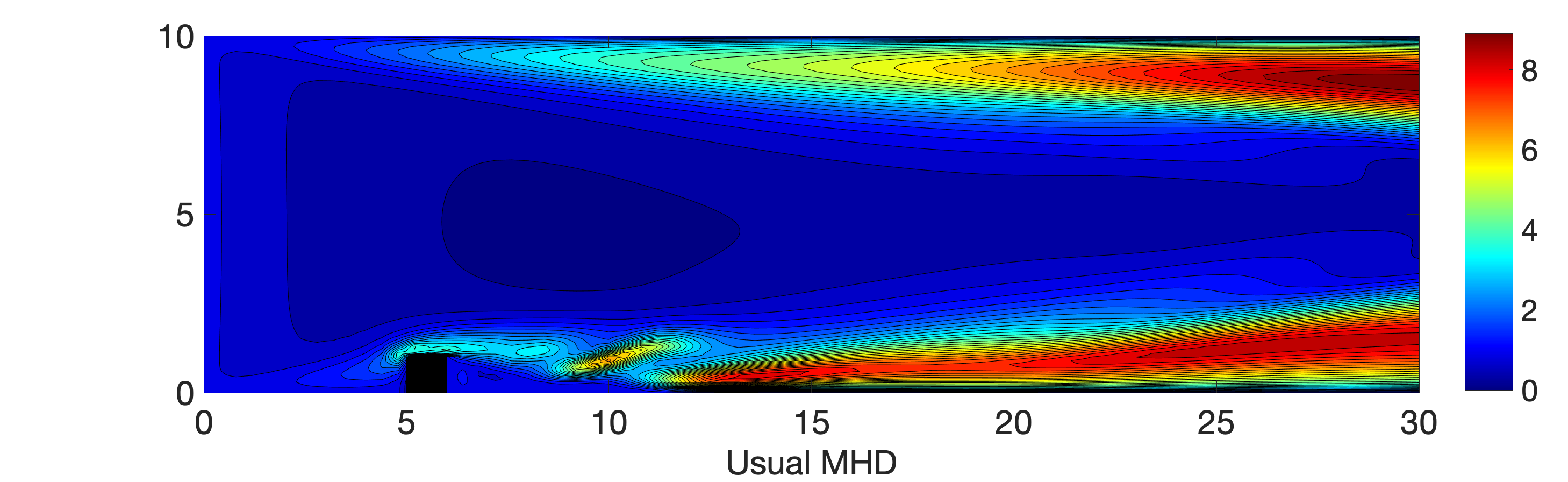}
	 	 	\caption{The magnetic field ensemble solution (magnetic field strength) at $T=40$ for MHD channel flow over a step with $\Delta t=1$, $s=0.001$, $\nu=0.001$, $\nu_m=0.01$, $\theta=1/9$, magnetic field $dof=1473898$, and magnetic pressure $dof=184563$.}\label{mag_ens_sol}
\end{center}
\end{figure}

\clearpage

\section{Conclusion and future works}

In this paper, we proposed, analyzed, and tested \textcolor{black}{a second order in time in practice, optimally accurate in space, decoupled and} efficient algorithm for MHD flow ensemble \textcolor{black}{computations}. \textcolor{black}{The second order temporal accuracy in practice} is a major improvement to the first order ensemble average algorithm proposed in \cite{MR17}. \textcolor{black}{The algorithm extends the breakthrough idea for efficient computation of flow ensemble for Navier-Stokes \cite{JL14} to MHD and combines with the breakthrough idea of Trenchea \cite{T14} to \textcolor{black}{construct} a decoupled stable scheme in terms of Els\"asser variables.} 
The key features to the efficiency of the algorithm are: (i) It is a stable decoupled method, split into two Oseen problems, \textcolor{black}{which are identical to assembled,} much easier to solve and can be solved simultaneously.  (ii) At each time step, all $J$ different linear systems share the same coefficient matrix, as a result, the storage requirement is reduced, a single assembly of the coefficient matrix is required instead of $J$ times, preconditioners need to be built once and can be reused. (iii) \textcolor{black}{It can take the advantage of the use of a block linear solver. (iv) No data restrictions are needed to avoid instability due to certain viscous terms.}
We proved the stability and second order convergence of the algorithm with respect to the timestep size. Numerical experiments were done on a unit square with a manufactured solution that verified the predicted convergence rates. Finally, we applied our scheme on a benchmark channel flow over a step problem and showed the method performed well. \textcolor{black}{Though, in our analysis, a timestep restriction appears which was not observed in the numerical experiments.}

In this paper, we considered the flow ensemble subject to the \textcolor{black}{slightly} different initial conditions and \textcolor{black}{forcing} functions, we plan to investigate flow ensemble behavior where the viscosities, and the boundary conditions involve uncertainties. \textcolor{black}{In the numerical experiments, no timestep restriction was observed, and thus further investigation is needed to find unconditional stability of the scheme.} \textcolor{black}{The recent idea \cite{gardner2020continuous} of  a continuous data assimilation algorithm for a velocity-vorticity formulation of the Navier-Stokes equations can be applied to the MHD flow ensemble.} To reduce the computational cost, reduced order modeling (ROM) for the ensemble MHD flow computation will be the future research avenue. Recently, it has been shown the data-driven filtered ROM for flow problem \cite{XMRI17} works well for the complex system. To reduced computation cost further to simulate an ensemble MHD system as well as more accurate results, it is worth exploring in ROM with physically accurate data \cite{MRI18}. We also plan to apply the recent advances \cite{gunzburger2019evolve} of an evolve-filter-relax based stabilization of ROMs for uncertainty quantification of the MHD flow ensemble using stochastic collocation method.

\textcolor{black}{The finite element simulations of the Maxwell equations using nodal based elements often produce cancellation errors, interface problems \cite{albanese1993analysis}, and spurious modes \cite{bossavit1990solving,sun1995spurious} that cause unwanted and unphysical solutions. For magnetic field, the tangential components are continuous across inter-element boundaries but it is not necessary for the normal components. By the nature of finite-element interpolants, some nodal vectorial elements enforce the continuity of the normal component across the interface which is not required by the physics. It is now well established that the edge elements, are more appropriate for the finite element discretization of the Maxwell equations \cite{albanese1997finite, beck1999hierarchical, Bossavit1988rationale}, which only enforces the tangential continuity of the magnetic field across the interfaces. MHD flow ensemble simulations with N\'ed\'elec's edge element \cite{nedelec1980mixed} will be the next research direction.
} 
\\\\
\textbf{Acknowledgment.} The author thanks Dr. Leo G. Rebholz for his constructive comments and suggestions that greatly improved the manuscript.

\bibliographystyle{plain}
\bibliography{High_order_MHD_ensemble}
\end{document}